\documentclass[12pt]{amsart}
\usepackage[margin=3cm]{geometry}
\usepackage{amsmath, amssymb, amsthm}
\usepackage{graphicx, color}
\usepackage[usenames,dvipsnames]{xcolor}
\usepackage{hyperref}
\usepackage{subfigure}
\usepackage{paralist}
\usepackage{comment}
\usepackage{mathtools}
\usepackage{makerobust}
\MakeRobustCommand\rotatebox

\newcommand{\RSE}{E}
\newcommand{\RSW}{O}

\newcommand{\AO}{AO}
\newcommand{\RP}{RP}
\newcommand{\PF}{PF}

\newcommand{\GR}[1]{\Sn{#1}(4231, 35142, 42513, 351624)}

\newcommand{\id}{\iota}

\newcommand{\mat}{\operatorname{mat}}
\newcommand{\m}{\operatorname{M}}
\newcommand{\M}{\mathcal{M}}
\newcommand{\X}{\mathcal{X}}
\newcommand{\inv}{\ell}

\newcommand{\sammag}{\rotatebox[origin=c]{180}{{\rm L}}}
\newcommand{\scrammag}{\rotatebox[origin=c]{180}{\scriptsize {\rm L}}}

\newcommand{\sLe}{\reflectbox{{\rm L}}}

\newcommand{\sGamma}{\Gamma}
\newcommand{\sEl}{{\rm L}}
\newcommand{\uu}{{\bf u}}
\newcommand{\vv}{{\bf v}}
\newcommand{\con}{\operatorname{con}}
\newcommand{\conw}{\con(w)}

\newcommand{\F}{\mathbf{F}}
\newcommand{\GL}{\mathrm{GL}}
\newcommand{\Sn}[1]{\mathfrak{S}_{#1}}
\newcommand{\Tnn}{Gr^{\geq 0}_{k,n}(\mathbb{R})}
\newcommand{\CC}{\mathbb{C}}

\newtheorem{prop}{Proposition}[section]
\newtheorem{proposition}[prop]{Proposition}
\newtheorem{cor}[prop]{Corollary}

\newtheorem{theorem}[prop]{Theorem}

\newtheorem*{mainthmmat}{Theorem~\ref{thm:matvspoin}}
\newtheorem*{mainthm}{Theorem~\ref{theorem:AO-RA}}

\newtheorem*{corLe}{Corollary~\ref{corollary:AO-Le}}

\newtheorem{lemma}[prop]{Lemma}
 
\newtheorem{prob}[prop]{Problem}
\newtheorem{conjecture}[prop]{Conjecture}

\theoremstyle{remark}

\newtheorem{example}[prop]{Example}

\newtheorem{remark}[prop]{Remark}

\newtheorem{definition}[prop]{Definition}

\begin{document}
\title{Combinatorics of diagrams of permutations}

\author[J.B. Lewis and A.H. Morales]{
Joel Brewster Lewis (University of Minnesota)\\
Alejandro H. Morales (University of California, Los Angeles)
}

\thanks{JBL was supported by NSF grants DMS-1148634 and DMS-1401792}
\thanks{AHM was supported by a CRM-ISM Postdoctoral Fellowship}
\thanks{An extended abstract of this article appeared as \cite{LMFPSAC}.}

\maketitle

\begin{abstract}
There are numerous combinatorial objects associated to a Grassmannian permutation $w_\lambda$ that index cells of the totally nonnegative Grassmannian. We study several of these objects and their $q$-analogues in the case of permutations $w$ that are not necessarily Grassmannian. We give two main results: first, we show that certain acyclic orientations, rook placements avoiding a diagram of $w$, and fillings of a diagram of $w$ are equinumerous for all permutations $w$.  Second, we give a $q$-analogue of a result of Hultman--Linusson--Shareshian--Sj\"ostrand by showing that under a certain pattern condition the Poincar\'e polynomial for the Bruhat interval of $w$ essentially counts invertible matrices over a finite field avoiding a diagram of $w$.  In addition to our main results, we include at the end a number of open questions.
\end{abstract}

\section{Introduction}

In his study \cite{AP} of the totally nonnegative Grassmannian $\Tnn$,
Postnikov introduced a ``zoo''
of combinatorial objects that parametrize cells of the
matroidal decomposition of $\Tnn$.  This decomposition refines the
Schubert decomposition $Gr_{k, n} = \bigcup_{\lambda \subset \langle (n -
  k)^k \rangle} \Omega_\lambda$, and the members of the zoo are most
easily identified with the \emph{Grassmannian permutations} $w =
w_\lambda \in \Sn{n}$.  Among the objects that appear in the zoo are the following four (whose precise definitions will be given later):
\begin{compactenum}[(i)]
\item The set of \emph{acyclic orientations} of the \emph{inversion graph} of $w$.

\item The set of \emph{placements of $n$ non-attacking rooks} on a board associated to $w$.

\item The set of certain \emph{restricted fillings} of a diagram associated to $w$.

\item The set of permutations below $w$ in the \emph{strong Bruhat order}.
\end{compactenum}

\noindent Work of Postnikov establishes the following result.
\begin{theorem}[{\cite[Thm.\ 24.1]{AP}}] 
\label{thm:APGrass}
For a Grassmannian permutation $w_{\lambda}$ in $\Sn{n}$, the sets above are equinumerous.
\end{theorem}

This theorem naturally raises the following question:

\begin{prob} \label{prob:main}
 Characterize the relation among these sets when $w$ is not Grassmannian.
\end{prob}

In the rest of this introduction we give background on previous work and a summary of
our own results towards answering this problem and its refinements.

\subsection{Definitions}

We begin by giving the definitions of the terms in the preceding paragraphs, which will be used throughout this paper.  Several definitions are illustrated in Figure~\ref{fig:exgrass}.

The pair $(i, j)$ is said to be an {\bf inversion} of the permutation $w \in \Sn{n}$ if $1\leq i<j \leq n$ and $w_i>w_j$.  
The {\bf inversion graph} $G_w$ of $w$ is the graph with vertex set
$[n]:= \{1,2,\ldots,n\}$ and with edges given by the inversions of $w$.  We consider the vertices to be ordered with smaller vertices to the \emph{left} or \emph{earlier} and larger vertices to the \emph{right} or \emph{later}.  
An {\bf acyclic orientation} of a graph $G$ is an orientation of the edges of $G$ so that the oriented graph has no directed cycles.  The number of acyclic orientations of $G$ is denoted $\AO(G)$.

A \textbf{diagram} (or \textbf{board}) is a finite subset of $\mathbb{Z}_{> 0} \times \mathbb{Z}_{> 0}$.
The {\bf south-east (SE) diagram} $\RSE_w$ (respectively, {\bf south-west (SW) diagram} $\RSW_w$) of the
permutation $w$ is the subset of 
$[n] \times [n]$ consisting of those
elements not directly to the south or east (respectively, south or west) of a nonzero
entry in the permutation matrix of $w$. (In \cite[\S 2.1]{LM}, the diagram
$\RSE_w$ is called the \emph{Rothe diagram} of $w$.) The size of $\RSE_w$ is the
number $\inv(w)$ of inversions of $w$, while the size of $\RSW_w$ is the number of {\bf anti-inversions}, i.e., the number of pairs $(i, j)$ such that $1 \leq i < j \leq n$ and $w_i < w_j$.  (This is also $\binom{n}{2} - \inv(w)$.)  Equivalently, $\RSE_w$ is the subset of $[n] \times [n]$ consisting of all pairs $(i,
w_j)$ such that $i< j$ and $w_j < w_i$ and $\RSW_w$ is the subset of $[n] \times [n]$ consisting of all pairs $(i,
w_j)$ such that $i<j$ and $w_j > w_i$.

A \textbf{rook placement} on a board $B$ is a set of cells (``rooks'') of $B$ such that no two lie in the same row or column.  For $B \subseteq [n] \times [n]$, we denote by $\RP(B)$ the number of rook placements of $n$ rooks \emph{avoiding} $B$, i.e., the number of placements of $n$ rooks on $([n]\times [n])\smallsetminus B$.

The {\bf (strong) Bruhat order} $\preceq$ of $\Sn{n}$ is the partial order on the symmetric group defined by the cover relations $w \prec w \cdot t_{ij}$ if $\ell(w\cdot t_{ij})=\ell(w)+1$ and $t_{ij}$ is the transposition that switches $i$ and $j$. 

We say that a permutation $w_{\lambda}$ in the symmetric group
$\Sn{n}$ on $n$ letters is a {\bf Grassmannian permutation} if it has at most one descent;
say the position of the descent is $k$. Each
such permutation is associated to a partition $\lambda$ inside the $k
\times (n-k)$ box $\langle (n - k)^k\rangle$ (i.e., a partition with at most
$k$ parts and largest part at most $n-k$). This correspondence can be
seen from the south-east diagram $\RSE_{w_{\lambda}}$, which is the Ferrers diagram of
$\lambda$ in French notation with possibly some columns in between,
see Example~\ref{examples of diagrams and graphs}. (Equivalently, this
correspondence comes from a certain {\em wiring diagram} of
$w_{\lambda}$, see \cite[Sec.~19]{AP} and Section~\ref{sec:fillings}.) 

A {\bf filling} of a diagram $D$ is an assignment of $0$s and $1$s to the elements of $D$.  A filling of $D$ is said to be {\bf percentage-avoiding}%
\footnote{
Postnikov \cite{AP} worked with a related family of fillings he called $\sLe$-diagrams, though percentage-avoidance is also implicit in some parts of his work.  The enumerative relationship between $\sLe$-diagrams and percentage-avoiding fillings is made explicit in Spiridonov \cite[\S 4]{AS} and is treated bijectively by Josuat-Verg\`es \cite[\S4]{MJV}.  For more on $\sLe$-diagrams, see Sections~\ref{sec:fillings} and~\ref{sec:pseudofillings}.
}
\cite{reiner-shimozono} if there are no four entries in $D$ at the vertices of a (axis-aligned) rectangle with either of the following fillings: \raisebox{-5pt}{\includegraphics{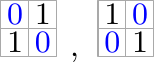}}.

\begin{example}
\label{examples of diagrams and graphs}
For the Grassmannian permutation $w=3412$, the SW diagram $\RSW_{3412}$ consists of the two elements
$(1,4), (3,2)$ and the SE diagram $\RSE_{3412}$ consists of the four
elements $(1,1)$, $(1,2)$, $(2,1)$, $(2,2)$. This $w$ has four inversions, at
positions $(1,3), (1,4), (2,3), (2,4)$, and two anti-inversions, at
positions $(1,2), (3,4)$. The inversion graph $G_{3412}$ has four
edges.  See Figure~\ref{ex:all3412}. 

For $w=3142$, the SW diagram $\RSW_{3142}$ and the SE diagram $\RSE_{3142}$ have three elements each and the inversion graph $G_{3142}$ has three edges.  This $w$ is not Grassmannian, and its SE diagram is not the diagram of a partition in French notation. See Figure~\ref{ex:all3142}.

\begin{figure}
\begin{center}
\subfigure[]{
\includegraphics{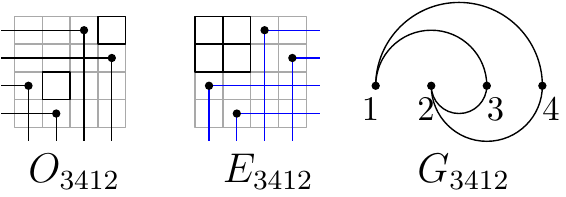}
\label{ex:all3412}
}
\quad
\subfigure[]{
\includegraphics{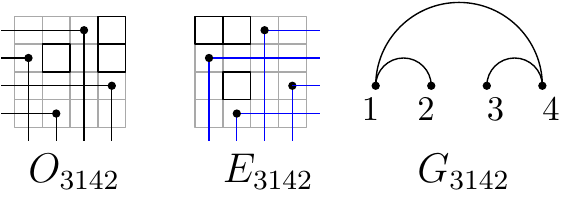}
\label{ex:all3142}
}
\caption{ The SW diagram, the SE diagram, and the
  inversion graph of the permutations (a) $3412$ and (b) $3142$.  For the Grassmannian permutation $3412$ (associated to the shape $\lambda = \langle2, 2\rangle$) we have
 $\AO(G_{3412})=\RP(O_{3412})=\#[1234, 3412]=14$.
}
\label{fig:exgrass}
\end{center}
\end{figure}
\end{example}

\subsection{Previous work}

The first result in the direction of Problem~\ref{prob:main} was the paper \cite{HLSS} of
Hultman--Linusson--Shareshian--Sj\"ostrand settling a conjecture of Postnikov \cite[Rem.~24.4]{AP}.
Their result explains the relation between the
number $\AO(G_w)$ of acyclic orientations of the inversion graph of $w$ and the
size of the Bruhat interval $[\id,w]$
.

\begin{theorem}[{\cite[Thm.~4.1, Cor.~5.7]{HLSS}}] 
\label{thm:HLSS}
For every permutation $w$ in $\Sn{n}$, we have $\AO(G_w) \leq \#[\id,w]$.  Furthermore, equality holds 
if and only if $w$ avoids the permutation patterns $4231$, $35142$, $42513$, and $351624$.
\end{theorem} 
The permutations on which equality is achieved are very special, and
will appear in the sequel.  We call them {\bf Gasharov--Reiner
  permutations} after their first appearance \cite{GasharovReiner} in
the literature.  (These permutations were recently enumerated by
Albert and Brignall \cite{AlbertBrignall}.)

Various authors have also explored $q$-analogues of the objects
defined above.  For example, \cite[Thm.~8.1]{HLSS} gives a $q$-analogue
of Theorem~\ref{thm:HLSS} involving the chromatic polynomial. Also, Oh--Postnikov--Yoo established the following result linking a $q$-analogue $A_w(q)$ (whose definition we omit) of $\AO(G_w)$ to the {\bf Poincar\'e polynomial} $P_{w}(q)=\sum_{u\preceq w}
q^{\ell(u)}$ (here the sum is over the permutations $u$ in the interval $[\id, w]$ of the Bruhat order).

\begin{theorem}[Oh--Postnikov--Yoo {\cite[Thm.\ 7]{OPY}}] 
\label{thm:OPY}
For a permutation $w$ in $\Sn{n}$, $A_w(q) = P_w(q)$ if and only if
$w$ avoids $3412$ and $4231$. 
\end{theorem}
\noindent Note that the classes of permutations of Theorem~\ref{thm:HLSS} and
Theorem~\ref{thm:OPY} differ, thus the $q=1$ version of the latter
does not imply the former. The class of permutations that appear in
Theorem~\ref{thm:OPY} are known as {\bf smooth permutations}.  They
have many very interesting properties (see \cite[\S 4]{AbeBilley}), of which we mention two here.

\begin{lemma}[Lakshmibai--Sandya \cite{LakshmibaiSandhya}, Carrell--Peterson \cite{CarrellPeterson}] \label{lemma:CP}
The following are equivalent for $w$ in $\Sn{n}$:
\begin{compactitem}
\item[(a)] The permutation $w$ is smooth.
\item[(b)] The Schubert variety $X_w$ associated to $w$ is smooth. 
\item[(c)] The Poincar\'e polynomial $P_w(q)$ is palindromic, that is, $P_w(q) = q^{\ell(w)}P_w(q^{-1})$.
\end{compactitem}
\end{lemma}

\subsection{New results}
In Section~\ref{sec:AOvsRP2}, we continue the study of the relationships
between the objects in Theorem~\ref{thm:APGrass} when $w$ is allowed to be an arbitrary
permutation. Our first result is a three-way equality involving a
suitable generalization of percentage-avoiding fillings. 
\begin{mainthm}
Given any permutation $w$ in $\Sn{n}$, the following are equal: the number $\AO(G_w)$
of acyclic orientations of the inversion graph $G_w$,  
 the number $\RP(\RSW_w)$ of placements of $n$ non-attacking rooks on the
  complement of the SW diagram $\RSW_w$, and the number of ``pseudo-percentage-avoiding fillings'' of the SE diagram $\RSE_w$ of $w$.
\end{mainthm}

One equality is proved in Section 2 and the other equality is proved
by Axel Hultman
in Appendix~\ref{sec:AOvsRP}. An immediate consequence of Theorem~\ref{thm:HLSS} and
Theorem~\ref{theorem:AO-RA} is that the number $\RP(\RSW_w)$ of rook placements has the same relation with
$\#[\id,w]$ as the number $\AO(G_w)$ of acyclic orientations (see
Corollary~\ref{corollary:RP-Bru}). 

In Section~\ref{sec:poinmat}, we study relations among
$q$-analogues of the objects described above. The first is the Poincar\'e polynomial $P_w(q)$, defined in the previous section, which is the natural $q$-analogue of the size $\#[\id, w]$ of the Bruhat interval below $w$.  
The other is a natural $q$-analogue of the rook placements avoiding the SW diagram of $w$.
\begin{definition}
Let $\F_q$ be the
  finite field with $q$ elements. Define $\mat_w(q)$ to be the number of
  $n\times n$ invertible matrices over $\F_q$  whose nonzero entries are in
  $\overline{\RSW_w}$. 
\end{definition}

It was shown in \cite[Prop.~5.1]{LLMPSZ} that $\m_w(q) := \mat_w(q)/(q-1)^n$ is an
enumerative $q$-analogue of $\RP(\RSW_w)$, in the sense that
\begin{equation} \label{eq:modq}
\m_w(q) \equiv \RP(\RSW_w) \; (=\AO(G_w)) \pmod{q-1}.
\end{equation}

Remarkably, the equality condition between $\m_w(q)$ and (an
appropriately rescaled version of) $P_w(q)$ is precisely the same as
between their values at $q = 1$ (as in Theorem~\ref{thm:HLSS}).  This settles part of a conjecture  \cite[Conj.\ 6.6]{KLM} of Klein and the present authors.

\begin{mainthmmat}
Let $w$ be a permutation in $\Sn{n}$. Then $\m_w(q) = q^{\binom{n}{2}+\inv(w)}P_{w}(q^{-1})$ if and only if $w$ avoids the
patterns $4231$, $35142$, $42513$, and $351624$.
\end{mainthmmat}

In Section~\ref{further remarks}, we give a large number of open
questions and other remarks.  Notably, in Sections~\ref{sec:fillings}
and~\ref{sec:pseudofillings}, we study additional relatives of
Postnikov's ``$\sLe$-diagrams''  and their relations with Bruhat intervals
(see Conjecture~\ref{conj:pseudofillpoin}) and acyclic orientations.  Our results include the following.
\begin{corLe}
If $w$ avoids $321$ then the number of $\sGamma$-diagrams on $\RSE_w$ is equal to the number of acyclic
orientations of the inversion graph of $w$.
\end{corLe}

Supplementary data and code for {\tt Sage} and {\tt Maple} are available
at the website\\
 \url{http://sites.google.com/site/matrixfinitefields/}.

\section{Acyclic orientations, rook placements, and fillings}
\label{sec:AOvsRP2}

The main result of this section is the following:

\begin{theorem}
\label{thm:AO-RP-PAF}
\label{theorem:AO-RA}
Given a permutation $w$ in $\Sn{n}$, the following are equal:
\begin{compactenum}[(i)]
\item the number $\AO(G_w)$ of acyclic orientations of the inversion graph of $w$,
\item the number $\RP(\RSW_w)$ of placements of $n$ non-attacking rooks on the
  complement of the SW diagram $\RSW_w$ of $w$,
\item the number of ``pseudo-percentage-avoiding fillings'' of the SE diagram $\RSE_w$ of $w$.
\end{compactenum}
\end{theorem}

As a corollary of Theorem~\ref{thm:AO-RP-PAF} and Theorem~\ref{thm:HLSS}, we have the following result.

\begin{cor}
\label{corollary:RP-Bru}
The number of placements of $n$ non-attacking rooks on $\overline{\RSW_w}$ equals
the number of permutations in the Bruhat interval $[\id, w]$ if and only if w avoids $4231$, $35142$, $42513$, and $351624$.
\end{cor}

The proof of the equality of $\AO(G_w)$ and $\RP(\RSW_w)$ is deferred
to Appendix~\ref{sec:AOvsRP}, where Axel Hultman gives an elegant
proof using some classic results from rook theory.  (An alternative,
longer proof may be found in the extended abstract \cite{LMFPSAC}; see
also Section~\ref{sec:bijective proofs}.) Then in Section~\ref{sec:AOvsPAF} we define the pseudo-percentage-avoiding fillings and complete the proof of Theorem~\ref{thm:AO-RP-PAF}.

\subsection{Bijection between acyclic orientations and pseudo-percentage-avoiding fillings} \label{sec:AOvsPAF}

Recall that $\RSE_w$ is the subset of $[n] \times [n]$ consisting of all pairs $(i, w_j)$ such that $i< j$ and $w_j < w_i$ (see Figure~\ref{fig:exgrass}, center panels).  In this section, we complete the proof of Theorem~\ref{thm:AO-RP-PAF} by establishing the equality of the number $\AO(G_w)$ of acyclic orientations of the inversion graph of $w$ with the number of \emph{pseudo-percentage-avoiding fillings} of the SE diagram $\RSE_w$, which we define now.

\begin{definition}
\label{PPAF definition}
Given a permutation $w$, we say that a filling $A$ of $\RSE_w$ with $0$s and $1$s is a \textbf{pseudo-percentage-avoiding filling} if it satisfies the following conditions:
\begin{compactenum}[(i)]
\item $A$ is percentage-avoiding, i.e., if squares $(i, j)$, $(i', j)$, $(i, j')$ and $(i', j')$ are elements of $\RSE_w$ then we do not have $A_{i,j} = A_{i',j'} = 1$ and $A_{i', j} = A_{i, j'} = 0$, nor do we have $A_{i,j} = A_{i',j'} = 0$ and $A_{i', j} = A_{i, j'} = 1$;
\item if squares $(i, j)$, $(i', j)$ and $(i, j')$ are elements of $\RSE_w$ and square $(i', j')$ is an entry of $w$ (that is, $j' = w_{i'}$) then we do not have $A_{i,j} = 1$ and $A_{i', j} = A_{i, j'} = 0$, nor do we have $A_{i,j} = 0$ and $A_{i', j} = A_{i, j'} = 1$.
\end{compactenum}
These forbidden patterns can be represented
by the images \raisebox{-7pt}{\includegraphics{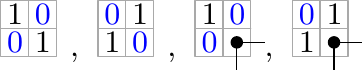}} where the solid dot indicates an
entry of the permutation.
\end{definition}

The main result of this section is the following:
\begin{prop} \label{proposition:AO-PAF}
Given any permutation $w$, the number of pseudo-percentage-avoiding fillings of $\RSE_w$ is equal to the number of acyclic orientations of $G_w$.
\end{prop}

We will need the following property of inversion graphs, whose (easy) proof is
left to the reader.

\begin{remark}\label{remark:edgesinvg}
Given a permutation $w$ in $\Sn{n}$ with inversion graph $G_w$ and vertices
$1\leq i <j <k \leq n$, we have that
\begin{compactenum}[(i)]
\item if $\{i,j\}$ and $\{j,k\}$ are edges of $G_w$ then
  $\{i,k\}$ is an edge of $G_w$, and
\item if $\{i,k\}$ is an edge of $G_w$ then at least one of $\{i,j\}$ and
  $\{j,k\}$ is an edge of $G_w$.
\end{compactenum}
\end{remark}

\begin{proof}[Proof of Proposition~\ref{proposition:AO-PAF}]
Call a cycle in an orientation of an inversion graph \emph{alternating} if its edges alternate between being directed to the right and to the left.  (In particular, only cycles of even length may be alternating.)

Consider any filling $f$ of $\RSE_w$. Recall that the elements $(i,w_j)$ of $\RSE_w$
are in correspondence with the inversions $(i,j)$ of $w$ and in turn with
the edges $\{i,j\}$
of $G_w$. In the inversion graph $G_w$, direct edges
corresponding to entries filled
with $1$ to the right and edges corresponding to entries filled with $0$ to the left.
One has immediately that $f$ contains a percentage pattern if and only if the corresponding orientation of $G_w$ contains an alternating $4$-cycle, and $f$ contains a pseudo-percentage pattern (extended using an entry of
$w$) if and only if the orientation contains a (directed) $3$-cycle; see \autoref{fig:patts2cycles}. Thus, it suffices to show that an orientation of an inversion graph is acyclic if and only if it contains no $3$-cycles and no alternating $4$-cycles.  One implication is obvious.

\begin{figure}
\begin{center}
\includegraphics{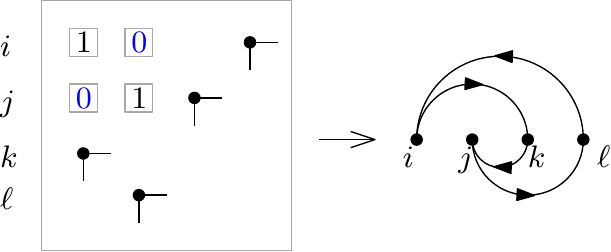}
\qquad  \qquad
\includegraphics{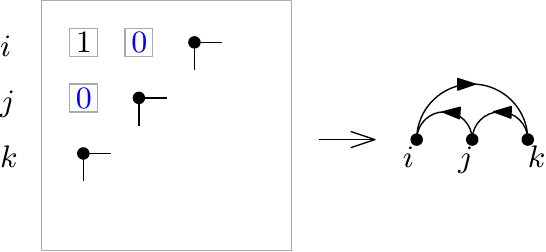}
\caption{Given a filling $f$ of $\RSE_w$ and a cell $(i, w_j) \in \RSE_w$, the edge of $\{i,j\}$ of $G_w$ is oriented to the right if $f(i,w_j) = 1$ and to the left if $f(i, w_j) = 0$. Left: a percentage pattern corresponds to an alternating $4$-cycle.  Right: a pseudo-percentage
  pattern corresponds to a $3$-cycle.}
\label{fig:patts2cycles}
\end{center}
\end{figure}

For the other direction, we wish to show that in every
orientation of $G_w$ with a directed cycle, there is a $3$-cycle or alternating $4$-cycle. 
Choose an orientation of $G_w$ that contains a directed cycle $C$.  We show that if $C$ is not a $3$-cycle or an alternating $4$-cycle then there is a cycle whose length is strictly less than that of $C$; this finishes the proof. 

Suppose that $C$ contains a chord, i.e., there is an edge of $G_w$
joining two vertices in $C$ that is not an edge of $C$.  In this case,
no matter which way one orients the chord, one produces a directed
cycle of strictly shorter length than $C$, as desired. 

Observe that if $C$ is not alternating then it necessarily contains a chord: if there are edges $a\to b$ and $b \to c$ of $C$ with $a < b < c$ or $a > b > c$ then by
Remark~\ref{remark:edgesinvg}(i) the inversion graph contains the edge $\{a, c\}$, a chord of $C$.  So we may suppose that $C$ is alternating and of length at least $6$.

Let $i_0$ be the leftmost vertex of $C$, and write $C = i_0 \to i_1 \to \ldots \to i_{2m - 1} \to i_{2m} = i_0$.  Choose $k > 0$ minimal so that $i_{2k + 2} < i_{2k}$.  From the choice of $k$ and the fact that $C$ is alternating it follows that $i_{2k - 2} < i_{2k} < i_{2k - 1}$ and $i_{2k + 2} < i_{2k} < i_{2k + 1}$.  We have two possibilities: first, if $i_{2k + 2}$ lies between $i_{2k - 2}$ and $i_{2k}$ then it lies between the endpoints of the edge $i_{2k - 2} \to i_{2k - 1}$ and so by Remark~\ref{remark:edgesinvg}(ii) there must be a chord joining $i_{2k + 2}$ to one of $i_{2k - 2}, i_{2k - 1}$.  Alternatively, if $i_{2k + 2} < i_{2k - 2}$ then $i_{2k - 2}$ lies between the endpoints of the edge $i_{2k + 1} \to i_{2k + 2}$ and so there must be a chord joining $i_{2k - 2}$ to one of $i_{2k + 1},  i_{2k + 2}$.
\end{proof}


\section{$q$-analogues of rook placements and Bruhat intervals}

\label{sec:poinmat}

Theorems~\ref{theorem:AO-RA} and~\ref{thm:HLSS} show that the size of
the Bruhat interval $[\id, w]$ is equal to the number $\RP(\RSW_w)$ of
rook placements avoiding the south-west diagram $\RSW_w$ if and only
if $w$ avoids the permutation patterns $4231$, $35142$, $42513$,
$351624$.  In this section, we study a natural $q$-analogue of this result, using a recursive analysis based on that of \cite{HLSS}.

The analogue of $\#[\id,w]$ that we consider is the {\bf Poincar\'e polynomial} 
\[
P_w(t):=\sum_{u\preceq w} t^{\ell(u)},
\] 
where the order relation in the sum is the strong Bruhat order.  The
analogues $\m_w(q)$ of the number $\RP(\RSW_w)$ of rook placements that we
consider are the {\bf matrix counting function} and the {\bf
  normalized matrix-counting function} defined by
\[
\mat_w(q) := \#\M(n,\RSW_w) \quad \text{and} \quad \m_w(q) := \frac{1}{(q - 1)^n} \cdot \mat_w(q),
\]
where 
\[
\M(n, \RSW_w) := \{ n \times n \textrm{ invertible matrices over } \F_q \text{ with nonzero entries restricted to } \overline{\RSW_w}\}
\]
and $\F_q$ is the finite field with $q$ elements.\footnote{One could alternatively view $\m_w(q)$ as counting orbits $T\backslash \M(n, \RSW_w)$ of matrices under the action of the (split maximal) torus $T$ of diagonal matrices in $\GL_n(\F_q)$, and indeed all of our proofs could be rephrased in this context.}  Equation
\eqref{eq:modq} shows that these are indeed $q$-analogues of $\RP(\RSW_w)$.

The main result of this section, answering part of the conjecture \cite[Conj.~6.6]{KLM}, is the following.

\begin{theorem}
\label{thm:matvspoin}
Let $w$ be a permutation in $\Sn{n}$. Then 
\begin{equation}
\label{desired q-equation}
\m_w(q) = q^{\binom{n}{2}+\inv(w)}P_{w}(q^{-1})
\end{equation}
if and only if $w$ avoids the
patterns $4231$, $35142$, $42513$, and $351624$.
\end{theorem}

 The proof of the ``only if'' part of Theorem~\ref{thm:matvspoin} is
as follows: by Equation~\eqref{eq:modq} we have that $\m_w(1) \equiv
\RP(\RSW_w) \pmod{q-1}$,
while by the definition of $P_w(q)$ we have that $q^{\binom{n}{2}+\inv(w)}P_{w}(q^{-1}) \mid_{q=1} =
\#[\id,w]$. If $w$ contains one of the patterns $4231$, $35142$,
$42513$, or $351624$ then $\RP(\RSW_w) \neq \#[\id,w]$ by Corollary~\ref{corollary:RP-Bru}.  Therefore for such $w$ the expressions $\m_w(q)$ and $q^{\binom{n}{2}+\inv(w)}P_{w}(q^{-1})$ cannot be equal for sufficiently large $q$.

The ``if'' part of the proof of Theorem~\ref{thm:matvspoin} is shown by induction,
and the rest of this section is devoted to its proof. Let $\GR{n}$ be the set of permutations $w$ in $\Sn{n}$ avoiding the
patterns $4231$, $35142$, $42513$, and $351624$, i.e., the
Gasharov--Reiner permutations. In \cite[\S 5]{HLSS}, the authors define two special kinds of descents called \emph{heavy} and \emph{light reduction pairs}.  We recall their definition here.

\begin{definition}
\label{definition of reduction pairs}
Suppose $w$ is a permutation with a descent formed by the entries $y=(i,w_i)$ and $x=(i+1,w_{i+1})$.
We call this descent a
{\bf light reduction pair} if 
\begin{compactitem}
\item there is no entry $(j,w_j)$ with $j<i$ and $w_j>w_{i}$, and
\item there is no entry $(j,w_j)$ with $j>i+1$ and $w_{i+1} < w_j < w_{i}$.
\end{compactitem}
This is illustrated in Figure~\ref{light}.
We call this descent a
{\bf heavy reduction pair} if
\begin{compactitem}
\item there is no entry $(j,w_j)$ with $j>i+1$ and $w_j<w_{i+1}$,
\item there is no entry $(j,w_j)$ with $j<i$ and $w_j>w_{i}$, and
\item there is an index $k$ with $w_i\leq k \leq w_{i+1}$ such
    that there is no entry $(j,w_j)$ with $j<i$ and $w_{i+1}<w_j\leq k$ or with
  $j>i+1$ and $k<w_j<w_{i}$. 
\end{compactitem}
This is illustrated in Figure~\ref{heavy}.  
\end{definition}

\begin{figure}
\begin{center}
\subfigure[]{
\includegraphics{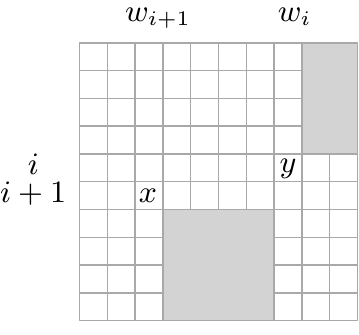}
\label{light}
}
\quad
\subfigure[]{
\includegraphics{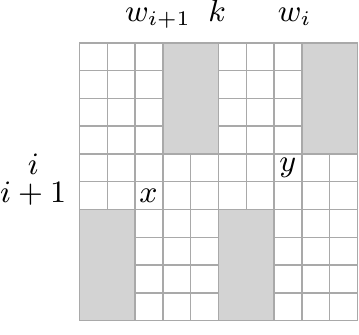}
\label{heavy}
}
\caption{(a) A light reduction pair, and (b) a heavy reduction pair. The
  gray areas have no entries $(j,w_j)$ of $w$.}
\label{redpairs}
\end{center}
\end{figure}

In \cite[Prop.\ 5.6]{HLSS}, it was shown that one can always find a reduction
pair in a permutation $w$ in $\GR{n}$.  
\begin{prop}[{\cite[Prop.\ 5.6]{HLSS}}] 
\label{prop:alwaysredpair}
Let $w\neq \iota$ be in $\GR{n}$. Then either the first descent of $w$ or the first descent of $w^{-1}$ is a reduction pair.

\end{prop}
Further, Hultman \emph{et al.} gave recursions for the size of the Bruhat
interval below Gasharov--Reiner permutations using the structure
imposed by the reduction pairs.  In the following sections, we extend
this work by giving recursions for the Poincar\'e polynomial and matrix
counting function of Gasharov--Reiner permutations.  
Thus, we will establish by induction that the Poincar\'e
polynomials and matrix counts are essentially equal in this case.

\subsection{Recursions for permutations with heavy reduction pairs}

In this section, we consider the case that the first descent of $w$ is
a heavy reduction pair.  In order to introduce our result, we must introduce some notation.
Given a permutation $w = w_1 \cdots w_n$ whose first descent is a
heavy reduction pair in position $i$, let $j$ be minimal such that
$w_j > w_{i + 1}$ and define $v = v(w)$ to be the permutation in
$\Sn{n - 1}$ that satisfies the order-isomorphism 
\begin{equation}
\label{definition of v}
v(w)\,\,\cong\,\, w_1 \, w_2 \, \cdots \, w_{j - 1} 
\quad
w_{j + 1} \, w_{j + 2} \, \cdots \, w_{i}
\quad w_j \quad   
w_{i + 2} \, w_{i + 3} \, \cdots \, w_{n}.
\end{equation} 
\begin{figure}
\begin{center}
\includegraphics[scale=1.25]{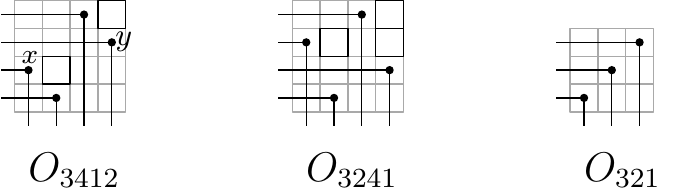}
\caption{South-west diagrams of 
$w=3412$, $s_2w=3142$, $v(w)=321$.}
\label{fig:O_125634}
\end{center}
\end{figure}
(The crucial properties of $v$ for our discussion are proved in
Propositions~\ref{prop:v and w/e have the same diagram}
and~\ref{prop:properties of v} below.)  In addition, we will make repeated use of the following operation on permutations.
\begin{definition}[Deletion in permutations and diagrams]
\label{def:deletion}
Suppose that $w = w_1 \cdots w_n \in \Sn{n}$ is a permutation and $y = (i, w_i)$ is an entry of $w$.  Then the result of \textbf{deleting} $y$ from $w$ is the permutation $w - y$ in $\Sn{n - 1}$ order-isomorphic to $w_1 \cdots w_{i - 1} w_{i + 1} \cdots w_n$.  

Similarly, for a diagram $D \subset [n] \times [n]$ and a pair $(i, j) \in [n] \times [n]$, {\bf deleting} $(i, j)$ from $D$ yields the diagram $D - (i, j) \subset [n - 1] \times [n - 1]$ that results from removing the $i$th row and $j$th column from $D$, and reindexing rows and columns as necessary.  (The two definitions can be easily seen to agree in the case that $D$ is a diagram with one entry in each row and column, $w$ is the associated permutation, and $y$ is an element of $D$.)
\end{definition}

The main result of this section is the following:

\begin{prop} \label{prop:HRP}
Let $w$ be in $\GR{n}$.  If the first descent of $w$, involving the entries $y = (i, w_i)$ and $x=(i+1,w_{i + 1})$, is a heavy reduction pair then 
\begin{equation}
\label{eq:Poincare HRP recursion}
P_w(t) = P_{s_i w}(t) + t^{\inv(w) - \inv(w - x)} P_{w - x}(t) + t^{\inv(w) - \inv(w - y)}P_{w - y}(t) - t^{\inv(w) - \inv(w - x - y)}P_{w - x - y}(t)
\end{equation}
and
\begin{equation}\label{eq:matrix HRP recursion}
\m_w(q) = \m_{s_iw}(q)+q^{n}\m_{v}(q),
\end{equation}
where $v$ is as in \eqref{definition of v}.
\end{prop}

\begin{example}
Let $w=3412 \in \Sn{4}$, whose first descent (at position $i=2$,
involving the entries $y = (2, w_2) = (2, 4)$ and $x=(3,w_3)=(3,1)$)
is a heavy reduction pair.  Then $s_2w=3142\in\Sn{4}$ and $w-x=231\in\Sn{3}$, $w-y=312 \in \Sn{3}$ and
$w-x-y=21 \in \Sn{2}$. One can compute the Poincar\'e polynomials
\begin{align*}
P_{3412}(t) &=t^4+4t^3+5t^2+3t+1,\\
P_{3142}(t) &=t^3+3t^2+3t+1,\\
P_{231}(t) = P_{312}(t) &=t^2+2t+1,\\
P_{21}(t) &= t+1
\end{align*} 
and verify that they satisfy the relation
\[
P_{3412}(t) = P_{3142}(t) + t^{4 - 2}\cdot P_{312}(t) + t^{4 - 2} \cdot P_{231}(t) - t^{4 - 1}\cdot P_{21}(t).
\]
For the matrix counts we have that $v=321 \in
\Sn{3}$ and one can compute
\begin{align*}
\m_{3412}(q) &= q^6(q^4+3q^3+5q^2+4q+1),\\
\m_{3142}(q) &= q^6(q^3+3q^2+3q+1),\\
\m_{321}(q)  &= q^3(q^3+2q^2+2q+1)
\end{align*} 
and verify that they satisfy the relation
\[
\m_{3412}(q) = \m_{3142}(q) + q^4 \m_{321}(q). 
\]
\end{example}

\subsubsection{Proof of Equation \eqref{eq:Poincare HRP recursion}}
\label{proof HRP for Poincare}

Given a Gasharov--Reiner permutation $w$ whose first descent, involving the entries $y = (i, w_i)$ and $x = (i + 1, w_{i + 1})$, is a heavy reduction pair, the argument of Hultman \emph{et al.} leading up to \cite[Eq.~(3)]{HLSS} establishes that the Bruhat interval $[\id, w]$ decomposes as the union of the following sets:
\begin{compactitem}
\item the Bruhat interval $[\id, s_i w]$, 
\item the set $S_x= \{u \in [\id, w] \mid u_{i + 1} = i + 1\}$ whose
  elements have an entry at $x$, and
\item the set $S_y = \{u \in [\id, w] \mid u_{i} = i\}$ 
  whose elements have an entry at $y$.
\end{compactitem}
Moreover, we may rephrase several of their observations as follows: they establish that $[\id, s_i w]$ is disjoint from $S_x$ and
$S_y$; that the maps $u \mapsto u - x$ and $u \mapsto u - y$ are
bijections respectively between $S_x$ and the Bruhat interval $[\id, w
- x]$ in $\Sn{n - 1}$ \cite[Lem.~2.1]{Billey} and between $S_y$ and $[\id, w - y]$; and
similarly that the map $u \mapsto u - x - y$ is a bijection between
$S_x \cap S_y$ and $[\id, w - x - y]$.  
Moreover, it follows from Sj\"ostrand's
result \cite[Thm.~4]{Sj} that every permutation $u \in [\id, w]$
satisfies $u_j < w_i$ for $j < i$ and $u_j > w_{i + 1}$ for $j > i + 1$.  
Consequently, among the permutations 
$u \in S_y$, the $i$th entry is always involved in
exactly the same number of inversions\footnote{As it happens, this number
$\inv(w) - \inv(w - y)$ is equal to $w_i - i$: there are $w_i - 1$
entries of $u$ smaller than $u_i = w_i$, of which $i - 1$ occur in $u$
before the $i$th position, leaving $w_i - i$ to occur after the $i$th
position; and none of the entries in $u$ before the $i$th position are
larger than $w_i$. One can make a similar computation with $y$
replaced by $x$.}, and similarly for those permutations with an
entry at position $x$. Putting everything together, we have
\begin{align*}
P_w(t) & = \sum_{u \preceq w} t^{\inv(u)} \\
& = \sum_{u \preceq s_iw} t^{\inv(u)} + \sum_{u \in S_x} t^{\inv(u)} +
\sum_{u \in S_y} t^{\inv(u)} - \sum_{u \in S_x\cap S_y} t^{\inv(u)}\\
& = P_{s_i w}(t) + \sum_{u \preceq w - x} t^{\inv(u) + \inv(w) -
  \inv(w - x)} + \sum_{u \preceq w - y} t^{\inv(u) + \inv(w) - \inv(w
  - y)} - \sum_{u \preceq w - x - y} t^{\inv(u) + \inv(w) - \inv(w - x
  -y)} \\
& = P_{s_i w}(t) +t^{\inv(w) - \inv(w - x)} P_{w - x}(t) +t^{\inv(w) -
  \inv(w - y)} P_{w - y}(t) - t^{\inv(w) - \inv(w - x - y)} P_{w - x - y}(t),
\end{align*}
as desired.

\subsubsection{Proof of Equation \eqref{eq:matrix HRP recursion}}

We get the desired recursion for $\m_w(q)$ by careful applications
of Gaussian elimination using the entry $(i + 1, w_i)$. Throughout the
proof it will be helpful to refer to Figure~\ref{fig:hrp}.  We begin by noting some properties of heavy reduction pairs that follow immediately from Definition~\ref{definition of reduction pairs}.
\begin{remark}
\label{hrp properties remark}
If the first descent of $w = w_1\cdots w_n$ is a heavy reduction pair in position $i$, and if $j$ is minimal so that $w_{i + 1} < w_j$, then
\begin{compactenum}[(i)]
\item $(w_1, \ldots, w_{j - 1}) = (1, \ldots, j - 1)$,
\item $(w_j, \ldots, w_i) = (w_i -i + j, w_i - i + j + 1, \ldots, w_i)$, and
\item $w_{i + 1} = j$.
\end{compactenum}
\end{remark}

\begin{figure}
\begin{center}
\includegraphics[height=5cm]{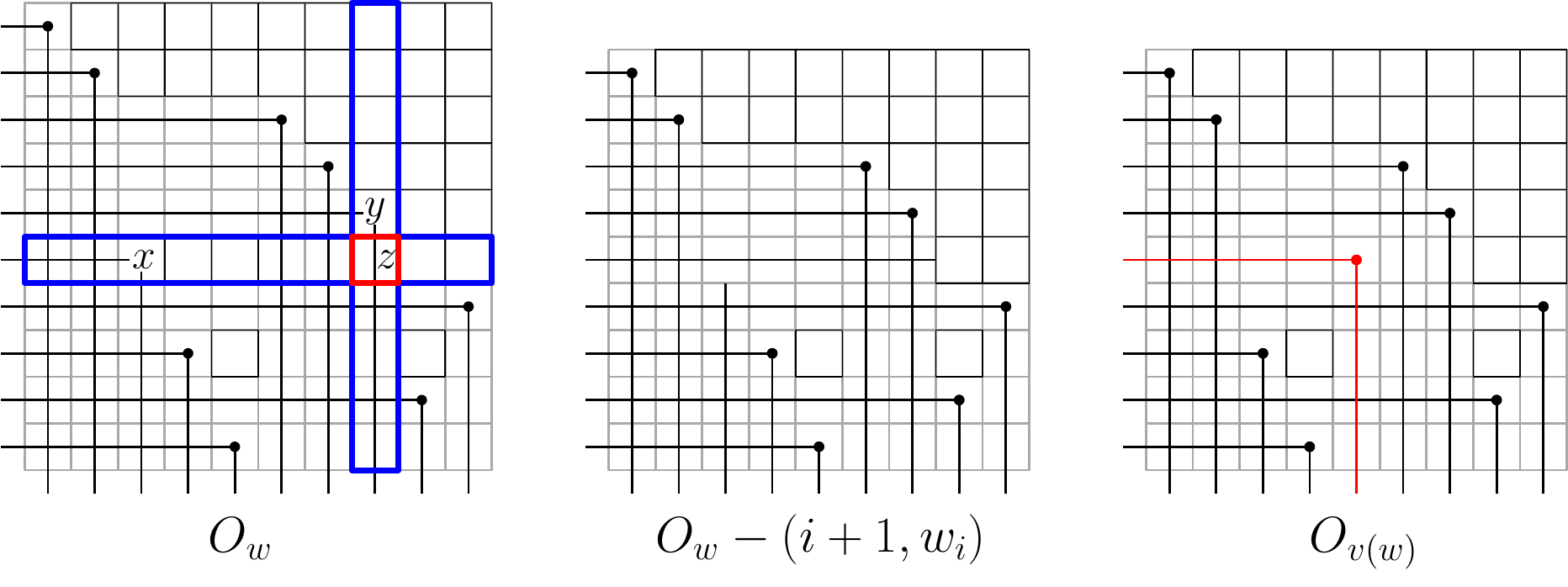}
\caption{Example of $\RSW_w$ for the permutation
  $w=1\,2\,6\,7\,8\,3\,10\,4\,9\,5$ whose first descent at position $i=5$ is a heavy
  reduction pair. Left: we perform Gaussian
  elimination on matrices in $\M(10,\RSW_w)$ with respect to the entry
  $z=(i+1,w_i)=(6,8)$. Center: the diagram after elimination. Right: the
  diagram $\RSW_{v(w)}$ for the permutation $v(w)$.}
\label{fig:hrp}
\end{center}
\end{figure}

\begin{prop} \label{prop:qelim}
Let $w$ be in $\Sn{n}$.  If the first descent of $w$, involving the entries $y = (i, w_i)$ and $x=(i+1,w_{i + 1})$, is a heavy reduction pair then 
\[
\m_w(q) = \m_{s_i w}(q) + q^n \cdot \#\M(n,\RSW_{w} - (i + 1, w_i))/(q - 1)^n,
\] 
where the deleted diagram $\RSW_{w} - (i + 1, w_i)$ is as in Definition~\ref{def:deletion}.
\end{prop}
\begin{proof}
Let $z=(i+1,w_i)$.  Since $O_{s_i w}$ equals the diagram $\RSW_w \cup
\{z\}$ with rows $i$ and $i+1$ switched, the
difference $\m_w(q) - \m_{s_i w}(q)$ is (up to the factor $(q - 1)^n$) the number of invertible matrices with support contained in $\overline{\RSW_w}$ having nonzero entry in position $z$.  We now examine the entries of $\RSW_w$ in the row and column of $z$.

Let $R$ be the union of the rows indexed by $\{1, 2, \ldots, i - 1\}$ and let $C$ be the union of the columns indexed by $\{w_{i + 1} + 1, \ldots, w_j - 1\} \cup \{w_i + 1, \ldots, n\}$.  It follows from Remark~\ref{hrp properties remark} and the definition of the SW diagram $\RSW_w$ that the entries of $\RSW_w$ in row $i + 1$ are exactly those in $C$ and the entries of $\RSW_w$ in column $w_i$ are exactly those in $R$, and that $\RSW_w$ is contained in $R \cup C$.  Consequently, if we superimpose row $i + 1$ with any row in $R$, the entries in $\RSW_w$ in row $i + 1$ cover those in the other row; and, similarly, if we superimpose column $w_i$ with any column from $C$, the entries in $\RSW_w$ in column $w_i$ cover those in the other column.

Now consider the set of matrices in $\M(n, \RSW_w)$ with nonzero entry
in position $z$.  Given such a matrix $A$, perform the following
operation: use Gaussian elimination with the nonzero entry in position
$z$ to kill the other nonzero entries in its row and column, then
delete the row and column of $z$.  The resulting matrix $B$ certainly
belongs to $\GL_{n-1}(\F_q)$.  Moreover, the analysis of the preceding
paragraph guarantees that no step in the elimination procedure disturbs any of the zero entries in positions given by $\RSW_w - z$, so in fact $B$ belongs to $\M(n - 1, \RSW_w - z)$.  Finally, given a matrix $B$ in $\M(n - 1, \RSW_w - z)$, one may reverse this process in precisely $(q - 1)q^n$ ways: first, choosing a nonzero entry for position $z$, then making appropriate row operations to fill in the $n$ entries in row $i + 1$ and column $w_i$ that do not belong to $R$ or $C$.  The resulting matrix belongs to $\M(n, \RSW_w)$ by the same analysis.  The result follows.
\end{proof}

\begin{prop}
\label{prop:v and w/e have the same diagram}
Let $w$ be in $\Sn{n}$.  If the first descent of $w$, involving the entries $y = (i, w_i)$ and $x=(i+1,w_{i + 1})$, is a heavy reduction pair then $\RSW_{w} - (i + 1, w_i)$ and $\RSW_{v(w)}$ are identical up to permutations of rows and columns. 
\end{prop}

\begin{proof}
Let $j$ be minimal such that $w_{i + 1} < w_j$.  By construction, the diagram $\RSW_{v - (i, v_i)} = \RSW_{v} - (i, v_i)$ is identical to the diagram that we get by removing the $i$th row and $w_{i + 1}$th column from $\RSW_w - (i + 1, w_i)$, as both are identical to the diagram that we get by removing the $i$th and $(i + 1)$st rows and $w_i$th and $w_{i + 1}$th columns from $\RSW_w$.  Thus, it suffices to check that the $i$th rows of $\RSW_{v}$ and $\RSW_w - (i + 1, w_i)$ are equal and that the $(w_j - 1)$th column of $\RSW_v$ is equal to the $w_{i + 1}$th column of $\RSW_w-(i + 1, w_i)$.

First, we consider the columns.  By Remark~\ref{hrp properties remark}, the $w_{i + 1}$th column of $\RSW_w$ contains exactly the $j - 1$ entries $(1, w_{i + 1})$, $(2, w_{i + 1})$, \ldots, $(j - 1, w_{i + 1})$, and so the $w_{i + 1}$th column of $\RSW_w-(i + 1, w_i)$ consists of these same $j - 1$ entries.  Similarly, applying Remark~\ref{hrp properties remark} and the definition of $v$, we see that the $(w_{j} - 1)$th column of $\RSW_{v}$ consists of the entries $(1, w_j - 1)$, \ldots, $(j - 1, w_j - 1)$, as needed.

Second, we consider the rows.  Since $i$ is the position of the first descent, 
$w_i$ is a left-to-right maximum of $w$.  Thus, the $i$th row of $\RSW_w$ consists of the $n-w_i$ elements $(i, w_i+1)$, $(i, w_i+2)$, \ldots, $(i, n)$, and no others.  Then the $i$th row of $\RSW_w - (i + 1, w_i)$ also consists of these $n-w_i$ boxes, each shifted one unit to the left.  In $v$, the entries $(j, w_j), \ldots, (i - 1, w_i - 1)$ do not form inversions with the entry $(i, w_j - 1)$, while the entries in columns $w_i, \ldots, n- 1$ occur in rows with indices larger than $i$.  Thus, the $i$th row of $\RSW_v$ consists of the same $n - w_i$ entries as the $i$th row of $\RSW_w - (i + 1, w_i)$.
%
\end{proof}

Finally, we may put these two propositions together to conclude the desired
recursion \eqref{eq:matrix HRP recursion} for matrix counts.

\subsection{Recursions for permutations with light reduction pairs}

In this section, we consider the case that the first descent of $w$ is
a light reduction pair. The main result of this section is the following one, which gives a pair of related recursions for Poincar\'e polynomials and matrix counts:

\begin{prop} \label{prop:LRP}
Let $w$ be in $\GR{n}$.  If the first descent of $w$, involving the entries $y = (i, w_i)$ and $x = (i + 1, w_{i + 1})$, is a light reduction pair then
\begin{equation}
\label{LRP prop equation for Poincare}
P_w(t) = P_{s_iw}(t) + t^{\inv(w) - \inv(w - y)}P_{w - y}(t)
\end{equation}
and
\begin{equation}
\label{LRP prop equation for matrices}
\m_w(q) = q\cdot \m_{s_iw}(q)+q^{n-1}\cdot \m_{w-y}(q).
\end{equation}
\end{prop}

\begin{example}
  With $w=3241$, the descent of $w$ in position $i=1$ is a light
  reduction pair with $y=(1,w_1)=(1,3)$ and $x=(2,w_2)=(2,2)$. We have
  $s_1w=2341$ and $w-y=231$. See \autoref{fig:O_3241} for the
  south-west diagrams of $w$, $s_1w$, and $w-y$.
One can compute the Poincar\'e polynomials
\begin{align*}
P_{3241}(t) &=t^4 + 3t^3+4t^2 + 3t + 1,\\
P_{2341}(t) &=t^3+3t^2+3t + 1,\\ 
P_{231}(t) & = t^2 + 2t+1,
\end{align*}
and verify that they satisfy the relation
\[
P_{3241}(t)  = P_{2341}(t) + t^{4 - 2}P_{231}(t). 
\]
Similarly, one can compute the matrix counts
\begin{align*}
\m_{3241}(q) &= q^6(q^4 + 3q^3+4q^2+3q+1),\\
\m_{2341}(q) &= q^6(q^3+3q^2+3q+1),\\
\m_{231}(q) &=  q^3(q^2+2q+1),
\end{align*}
and verify that they satisfy the relation
\[
\m_{3241}(q)  = q\m_{2341}(q) + q^3\m_{231}(q).
\]
\end{example}

\begin{figure}
\begin{center}
\includegraphics[scale=1.25]{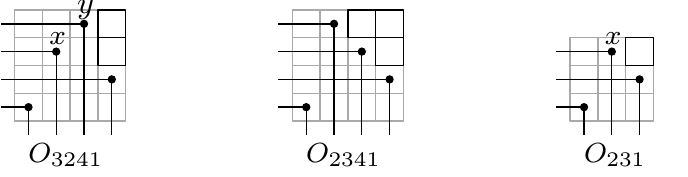}
\caption{South-west diagrams of $w=3241$, $s_1w=2341$ and $w-y=231$.}
\label{fig:O_3241}
\end{center}
\end{figure}

\subsubsection{Proof of Equation \eqref{LRP prop equation for Poincare}}
\label{proof of LRP for Poincare}

Given a Gasharov--Reiner permutation $w$ whose first descent, involving
the entries $y = (i, w_i)$ and $x = (i + 1, w_{i + 1})$, is a light
reduction pair, the argument of Hultman \emph{et al.} leading up to
\cite[Eq.~(1)]{HLSS} establishes that the Bruhat interval $[\id, w]$
decomposes as the disjoint union of the Bruhat interval $[\id, s_i w]$
and the set of permutations below $w$ that contain an entry at
position $y$.  In addition, the operation $u \mapsto u - y$ is a
bijection between the latter set and the interval $[\id, w - y]$ in the
Bruhat order on $\Sn{n - 1}$. To finish the proof of \eqref{LRP prop equation for Poincare}, it is enough to observe that, as in Section~\ref{proof HRP for Poincare}, Sj\"ostrand's result \cite[Thm.~4]{Sj} can be used to establish that these bijections respect the grading of the Bruhat order.

\subsubsection{Proof of Equation \eqref{LRP prop equation for
    matrices}}

We begin with a useful lemma for how matrix counts avoiding a diagram
behave when one adds an entry to the diagram.
For $D \subseteq [n] \times [n]$, let $\X(n,D)$ be the set of all
$n\times n$ matrices with entries in $\F_q$ and support avoiding
$D$. (Thus $\M(n,D) = \X(n,D) \cap \GL_n(\F_q)$.)

Let $D \subseteq [n] \times [n]$ and $y\in \overline{D}$. Given a matrix $B \in \X(n-1,D-y)$ and $a,b \in \F_q$, define $S_{a\to
  b}(B)$ to be the number of matrices $A \in \M(n,D)$ such that:
\begin{compactitem}
\item $A_y=a$,
\item when one removes from $A$ the row and column of $y$, the result is $B$, and
\item  the $n\times n$ matrix $A'$ defined by $A'_y=b$ and $A'_z=A_z$
  for $z\neq y$ is singular.
\end{compactitem}

We prove a general lemma, showing how to express the difference between two matrix counts in terms of these $S_{a \to b}(B)$.

\begin{lemma} \label{lem:qdiffmat}
Let $D\subseteq [n] \times [n]$ and $y \in \overline{D}$. Then 
\[
\#\M(n,D) - q\cdot \#\M(n,D \cup \{y\}) = \sum_{a \in \F_q^{\times}}\smashoperator[r]{\sum_{B \in \M(n-1, D - y)}}
\left( S_{a\to 0}(B) - S_{0\to a}(B)\right).
\]
\end{lemma}

\begin{proof}
First, we give a convenient interpretation of the term $q\cdot
\#\M(n,D \cup \{y\})$. This counts pairs $(a,A)$ where $a \in \F_q$ and $A \in
\M(n,D\cup \{y\})$.  We view this as setting $A_y\mapsto a$ in the
invertible matrix $A$, which might or might not yield an invertible matrix.

Second, given a matrix $A$ in $\M(n,D \cup \{y\})$, the submatrix $B$ that results from removing the row and column of $y$ has
support in $\overline{D-y}$ and may or may not be invertible. We
show that the difference $\#\M(n,D)-q\cdot \#\M(n,D\cup \{y\})$ cancels all the
terms where $B$ is not invertible.

Let $S_{a\to b} = \sum_{B \in \X(n-1,D-y)} S_{a\to b}(B)$ be the
number of matrices $A$ in $\M(n,D)$ with $A_y=a$ such that setting
$A_y \mapsto b$ (and leaving all other entries of $A$ unchanged) yields a singular matrix. Similarly, let $I_{a\to b}$ be the
number of matrices $A$ in $\M(n,D)$ with $A_y=a$ such that setting
$A_y \mapsto b$ yields an invertible matrix. We break down $\M(n,D)$ as
\[
\#\M(n,D) = \sum_{a\in \F_q} \left(S_{a \to 0} + I_{a \to 0}\right),
\]
and we break down $q\cdot \#\M(n,D\cup \{y\})$ as
\[
q\cdot \#\M(n,D\cup \{y\}) = \sum_{a \in \F_q} \left(S_{0\to a} + I_{0\to a}\right).
\]
Note that $S_{0\to 0} =0$, and that for all $a\in \F_q$ we have $I_{a\to 0} = I_{0\to a}$. Thus
\[
\#\M(n,D)-q\cdot \#\M(n, D\cup \{y\}) = \sum_{a \in \F_q^{\times}}
\left( S_{a\to 0} - S_{0\to a}\right).
\]
 
Next, consider an $n\times n$ matrix $A$ over $\F_q$, thinking of the entry $A_y$ as variable,
and let $B$ be the matrix obtained by removing from $A$ the row and column of
$y$. Then $\det(A) = \pm \det(B) \cdot A_y+ k$ for some $k\in
\F_q$. If $\det(A)$ is nonconstant when viewed as a function of $A_y$ then the linear coefficient $\det(B)$ is nonzero. Therefore if $B \in
\X(n-1,D-y)$ is singular and $a\neq 0$ then $S_{a\to 0}(B)=S_{0\to a}(B)=0$. Thus
\begin{align*}
\#\M(n,D)-q\cdot \#\M(n, D\cup \{y\}) & = \sum_{a\in \F_q^{\times}}
\smashoperator[r]{\sum_{B\in \X(n-1,D-y)}} \left( S_{a\to 0}(B) - S_{0\to a}(B)\right)\\
&= \sum_{a\in \F_q^{\times}}
\smashoperator[r]{\sum_{B\in \M(n-1,D-y)}} \left( S_{a\to 0}(B) - S_{0\to a}(B)\right).\qedhere
\end{align*}
\end{proof}

We note a few points about the diagrams $\RSW_w$ of permutations $w$
whose first descent is a light reduction pair.  They follow immediately from Definition~\ref{definition of reduction pairs} (see Figure~\ref{fig:lrp}).

\begin{remark} \label{rem:propLRP}
Suppose the first descent of $w\in \Sn{n}$ is in position $i$ and is a light reduction pair.  Then:
\begin{compactenum}[(i)]
\item the $i$th and $(i+1)$st rows of $\RSW_w$ have entries in
  exactly the same set of columns, namely those with indices $\{w_i+1,\ldots,n\}$; and
\item all the entries in the north-east rectangle $[1,i-1]\times
  [w_{i}+1,n]$ are in $\RSW_w$.
\end{compactenum}
\end{remark}

 It follows from Remark~\ref{rem:propLRP}(i) that
$\RSW_{s_i w} = \RSW_w \cup \{(i,w_i)\}$. Then by Lemma~\ref{lem:qdiffmat}
applied to $D = \RSW_w$ and $y=(i,w_i)$ we have
\begin{equation}
\label{applying lemma to SW diagram}
\mat_w(q) - q\cdot \mat_{s_iw}(q) = \sum_{a \in
  \F_q^{\times}}\smashoperator[r]{\sum_{B \in \M(n-1, \RSW_{w-y})}}
\left( S_{a\to 0}(B) - S_{0\to a}(B)\right).
\end{equation}

Fix a matrix $B$ in $\M(n-1,\RSW_{w-y})$. From $B$ we build matrices $A=\begin{bmatrix} a& \uu \\
\vv & B \end{bmatrix}$ where $a \in \F_q$,  $\uu$ is a row
vector in $\F_q^{n-1}$ whose first $r=w_i-1$ entries are free and the
rest set to zero; and
$\vv$ is column vector in $\F_q^{n-1}$ whose last $c=n-1-i$ entries
are free
and the rest set to zero. The motivation for this construction is that these matrices are simply rearrangements of matrices with
support avoiding $\RSW_w$. In particular, for any choice of $a, \uu, \vv, B$, the resulting matrix $A$ satisfies $(1,2,\ldots,i)\cdot A \cdot
(1,2,\ldots,w_i)^{-1} \in \X(n,\RSW_w)$; see
Figure~\ref{fig:lrp}. The determinant of such a matrix is
\begin{equation} \label{eq:detA}
\det(A) = a\det(B)  - \uu B^{-1} \vv.
\end{equation}
There are $q^{r+c}$ of these matrices of the form
$\begin{bmatrix} 0& \uu \\
\vv & B \end{bmatrix}$. Each of these is invertible or has rank
$n-1$. Let $N(B)$ be number of such matrices that have rank $n-1$,
so the remaining $q^{r+c}-N(B)$ matrices are invertible.

We proceed to compute the terms $S_{a\to 0}(B)$ and $S_{0\to a}(B)$.
\begin{compactitem}
\item By \eqref{eq:detA}, a matrix $A$ is counted in $S_{a\to 0}(B)$ if and only if 
  $a\det(B)\neq 0$ and $\uu B^{-1} \vv=0$. This in turn is equivalent to
  $a\neq 0$ and $\begin{bmatrix} 0 & \uu \\ \vv & B\end{bmatrix}$
  having  rank $n-1$. Thus for each $a\in \F_q^{\times}$
    the number of such cases is $S_{a\to 0}(B)=N(B)$. 
\item  By \eqref{eq:detA}, a pair $(a,A)$ is counted in  $S_{0\to a}(B)$  if and only if $a\det(B) =
  \uu B^{-1}\vv\neq 0$. This implies that $\begin{bmatrix} 0 & \uu
    \\ \vv & B \end{bmatrix}$
has rank $n$ and thus 
\[
S_{0\to a}(B) = \begin{cases}
q^{r+c}-N(B)  & \text{ if } a = \uu B^{-1}\vv/\det(B),\\
0 & \text{ otherwise.}
\end{cases}
\]
\end{compactitem}
Substituting into Equation~\eqref{applying lemma to SW diagram} gives
\begin{align}
\mat_w(q) - q\cdot \mat_{s_iw}(q) 
& = \smashoperator[r]{\sum_{B\in
    \M(n-1,\RSW_{w-y})}} (q-1)N(B) - \smashoperator[r]{\sum_{B\in
    \M(n-1,\RSW_{w-y})}} \left(q^{r+c}-N(B)\right) \notag \\
& = \smashoperator[r]{\sum_{B\in
    \M(n-1,\RSW_{w-y})}} \left( q\cdot N(B) - q^{r+c}\right). \label{eq:LRPintermeddiate}
\end{align}

Finally, we compute $N(B)$. This is the number of choices of $\uu$ and
$\vv$ such that $\uu B^{-1} \vv = 0$ and $\uu$ and $\vv$ have support
as described in the paragraph preceding \eqref{eq:detA}.  Let $\uu' =
\uu B^{-1}$. If $\uu'$ has support in the first $n-1-c$ entries then
$\uu'\vv =0$ for all $q^c$ choices of $\vv$. By Remark~\ref{rem:propLRP}(ii), the matrix $B$ has a zero
block matrix in its north-east corner of size
$(n-1-c)\times (n-1-r)$ (see Figure~\ref{fig:lrp}), and so every
vector $\uu'$ with support in the first $n-1-c$ entries is sent by $B$
to a vector $\uu'B$ with support in the first $r$ entries. Since $B$
is invertible this implies that each of the $q^{n-1-c}$ vectors $\uu'$ with
support in the first $n-1-c$ entries is the image under $B^{-1}$ of a
vector $\uu$ with support in the first $r$ entries.

For the remaining $q^r - q^{n-1-c}$ choices of $\uu$, the matrix $\uu B^{-1}$ has support
intersecting the last $c$ entries and so there are $q^{c-1}$ choices of
$\vv$ such that $\uu B^{-1}\vv =0$. From the preceding two paragraphs
it follows that 
\begin{equation} \label{eq:N(B)}
N(B) =  q^{n-1-c}\cdot q^c +(q^r - q^{n-1-c})\cdot q^{c-1} = q^{n-1}+
q^{r+c-1}-q^{n-2}.
\end{equation}
Finally combining this with \eqref{eq:LRPintermeddiate} yields 
\begin{align*}
\mat_w(q) - q\cdot \mat_{s_iw}(q) & = \sum_{B \in \M(n-1,\RSW_{w-y})}
\left(q(q^{n-1} + q^{r+c-1} +q^{n-2}) - q^{r+c}\right)\\
&= q^{n-1}(q-1) \mat_{w-y}(q).
\end{align*}
Dividing by $(q-1)^n$ gives \eqref{LRP prop equation for matrices}, as desired.

\begin{figure}
\begin{center}
\includegraphics[height=5cm]{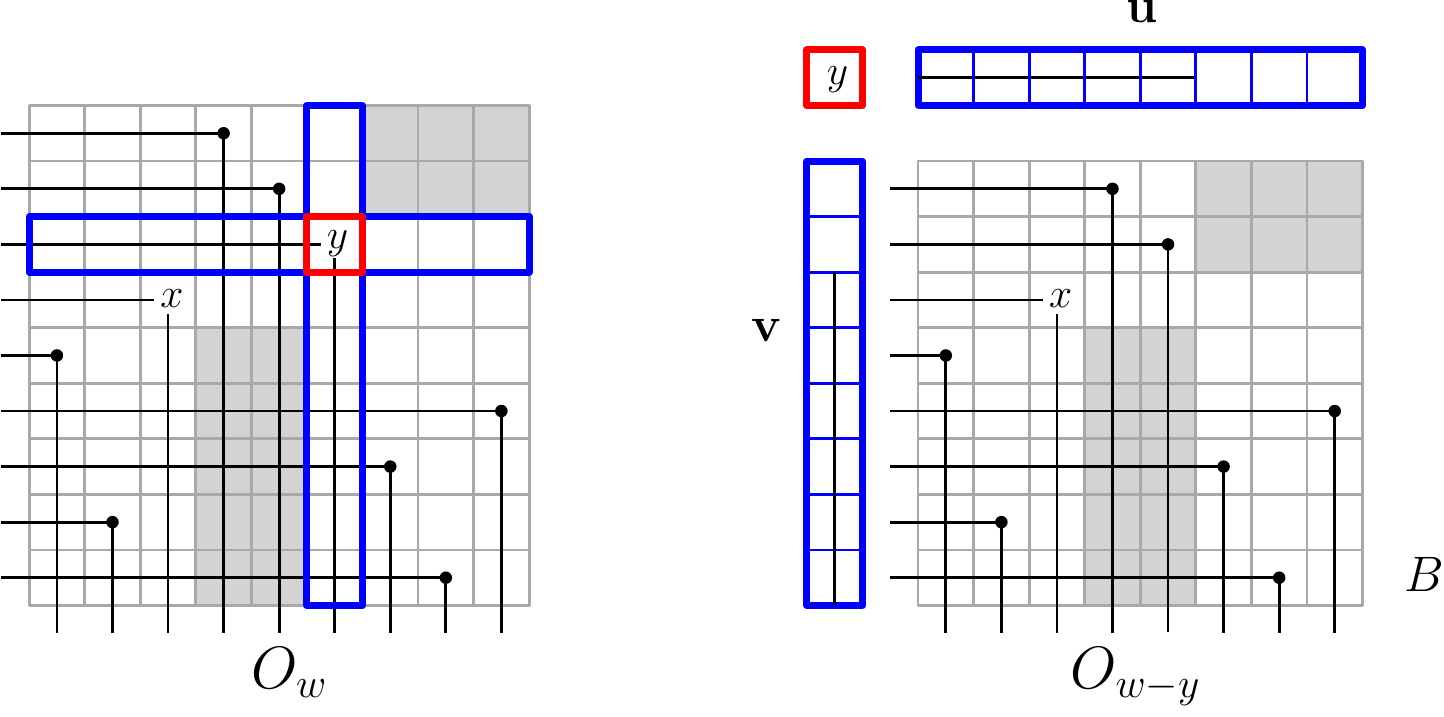}
\caption{Left: the diagram $\RSW_w$ for 
  $w=456319728$, whose first descent 
  is a light reduction pair. The key features of $\RSW_w$
  are that the two rows involved in the descent have 
  entries in the same columns and there is
  north-east rectangle (in gray) in $\RSW_w$. Right: rearrangement of
  the rows and columns of $\RSW_{w}$ yields the
  subdiagram $\RSW_{w-y}$ where $w-y=45318627$.}
\label{fig:lrp}
\end{center}
\end{figure}

\subsection{End of proof of Theorem~\ref{thm:matvspoin}}

Finally, in this section we put the preceding results together in order to
finish the inductive proof of the ``if'' part of
Theorem~\ref{thm:matvspoin}.  We induct simultaneously on the size (i.e., number of entries) and the length (i.e., number of inversions) of the permutation $w$, starting with the base case of identity permutations.

\begin{prop}
For the identity permutation $\id$, we have 
\[
\m_{\id}(q) = q^{\binom{n}{2}+\inv(\id)}P_{\id}(q^{-1})=q^{\binom{n}{2}}.
\]
\end{prop}

\begin{proof}
The result is trivial:
$(q - 1)^n \cdot \m_{\id}(q)=(q-1)^n q^{\binom{n}{2}}$ is the number of invertible $n\times n$
lower triangular matrices and $P_w(t)=1$.
\end{proof}

Now suppose that $w$ is a  Gasharov--Reiner permutation.  We have by
Proposition~\ref{prop:alwaysredpair} that either the first descent of $w$
or the first descent of $w^{-1}$ is a reduction pair.  (Note that
$w^{-1}$ is also Gasharov--Reiner.) For any permutation $w$ it is well-known that
$P_w(t)=P_{w^{-1}}(t)$, and by
\cite[Prop.\ 5.2(ii)]{KLM} the diagrams $\RSW_w$ and $\RSW_{w^{-1}}$ are
  rearrangements of each other and so  $\m_w(q)=\m_{w^{-1}}(q)$. Thus without loss of generality we may assume that the first descent of $w$ is a reduction pair.

\subsubsection{The case of a light reduction pair}

If the first descent of $w$, involving the entries $y = (i, w_i)$ and $x = (i + 1, w_{i + 1})$, is a light reduction pair then by
\eqref{LRP prop equation for matrices}
and induction we have
\begin{align*}
\m_w(q) & = q \cdot \m_{s_iw}(q) + q^{n - 1}\m_{w - y}(q) \\
& = q \cdot q^{\binom{n}{2} + \inv(s_iw)} P_{s_iw}(q^{-1}) + q^{n - 1}\cdot q^{\binom{n - 1}{2} + \inv(w - y)} P_{w - y}(q^{-1}).
\end{align*}
Then from \eqref{LRP prop equation for Poincare} it follows that
\begin{align*}
\m_w(q) & = q^{\binom{n}{2} + \ell(w)} \left(P_{s_iw}(q^{-1}) + q^{-(\inv(w) - \inv(w - y))} P_{w - y}(q^{-1})\right) \\
& = q^{\binom{n}{2} + \ell(w)} P_{w}(q^{-1}),
\end{align*}
as desired.

\subsubsection{The case of a heavy reduction pair}

Suppose the first descent of $w$ is a heavy reduction pair.  We first consider the case $i = j$, i.e., that $w_1 < \ldots < w_{i - 1} < w_{i + 1} < w_i$.  In this case, the definition of heavy reduction pairs implies that
\[
w = 1 \cdots (i - 1) \; w_i \; i \; \tau
\] 
for some permutation $\tau$ of $[n]\smallsetminus \{1,\ldots, i, w_i\}$.  Thus, we have that $s_iw = 1 \cdots i \, w_i \, \tau$ and that $w - x$ is order-isomorphic to $1 \cdots (i - 1) \, w_i\, \tau$, so $P_{s_iw}(t) = P_{w - x}(t)$.  Similarly, we have that $w - y$ is order-isomorphic to $1 \cdots i \, \tau$ and that $w - x - y$ is order-isomorphic to $1 \cdots (i - 1)\, \tau$, so $P_{w - x - y}(t) = P_{w - y}(t)$.  Then Equation~\eqref{eq:Poincare HRP recursion} reduces to 
\[
P_w(t) = (1 + t)\cdot P_{w - x}(t).
\]
Moreover, when $i = j$ we have $v(w) = w - x$, so Equation~\ref{eq:matrix HRP recursion} reduces to
\[
\m_w(q) = \m_{s_iw}(q) + q^n \cdot \m_{w - x}(q).
\]
Thus, by induction we have
\begin{align*}
\m_w(q) & = \m_{s_iw}(q) + q^n \cdot \m_{w - x}(q) \\
& = q^{\binom{n}{2} + \ell(s_iw)}P_{s_iw}(q^{-1}) + q^n \cdot q^{\binom{n - 1}{2} + \ell(w - x)}P_{w - x}(q^{-1}) \\
& = \left(q^{\binom{n}{2} + \ell(w) - 1} + q^{n + \binom{n - 1}{2} + \ell(w) - 1}\right)P_{w - x}(q^{-1}) \\
& = q^{\binom{n}{2} + \ell(w)}(q^{-1} + 1)P_{w - x}(q^{-1}) \\
& = q^{\binom{n}{2} + \ell(w)} P_w(q^{-1}),
\end{align*}
as desired.

Finally, we are left with the case that $i > j$.  In this case, the three permutations $w$, $s_iw$ and $w - y$ have first descents in positions $i$, $i -1$ and $i - 1$, respectively, and satisfy the hypotheses of Proposition~\ref{prop:HRP}.  Thus, applying Proposition~\ref{prop:HRP} to each of these permutations and rearranging gives
\begin{equation}
\label{eq:equations for vs}
\begin{aligned} 
q^{n}  \m_{v(w)}(q)     & = \m_w(q)      - \m_{s_iw}(q), \\
q^{n}  \m_{v(s_i w)}(q)  & = \m_{s_iw}(q) - \m_{s_{i-1}s_i w}(q),  \quad \textrm{ and }\\
q^{n-1} \m_{v(w - y)}(q)  &= \m_{w-y}(q)   - \m_{s_{i - 1}(w-y)}(q).
\end{aligned}
\end{equation} 
To make use of these equations, we need some basic properties of the permutation $v(w)$.

\begin{prop}
\label{prop:properties of v}
Suppose that $w \in \Sn{n}$ is a Gasharov--Reiner permutation whose first descent, involving the entries $(i, w_i)$ and $(i + 1, w_{i + 1})$, is a heavy reduction pair.  If $i > j$ then, with $v = v(w)$ as in \eqref{definition of v}, we have that $v \in \GR{n - 1}$ and the first descent of $v$ is in position $i - 1$ and is a light reduction pair.
\end{prop}
\begin{proof}
Certainly $v \in \Sn{n - 1}$.  To show that $v$ avoids the four bad patterns, it suffices to show that the sequence
\[
w' = w_1 \, w_2 \, \cdots \, w_{j - 1} 
\quad
w_{j + 1} \, w_{j + 2} \, \cdots \, w_{i}
\quad w_j \quad   
w_{i + 2} \, w_{i + 3} \, \cdots \, w_{n}
\]
(to which $v$ is order-isomorphic) avoids them.

Suppose that $w'$ contains one of the four forbidden patterns, and let $\mu$ be a subsequence of $w'$ order-isomorphic to one of these patterns.  We will derive a contradiction.

By Remark~\ref{hrp properties remark}(i), the entries $w_1, \ldots, w_{j - 1}$ are all smaller than all other entries of $w'$ and occur in increasing order, but none of the four forbidden patterns begins with its smallest element.  Thus, $\mu$ does not contain any of these entries.

Removing the entry $w_j$ from $w'$ leaves a sequence order-isomorphic to a subsequence of $w$.  Since $w$ avoids the four patterns in question, it follows that $\mu$ must contain the entry $w_j$.  The same is true if one removes (simultaneously) the entries $w_{j + 1}, w_{j + 2}, \ldots, w_i$ from $w'$, so $\mu$ must contain at least one of these entries.

Thus, the first descent of $\mu$ occurs between one of the values $w_{j + 1}, \ldots, w_i$ and $w_j$.  Therefore, by Remark~\ref{hrp properties remark}(ii), in the permutation order-isomorphic to $\mu$, the entries of the shortest prefix including the bottom of the first descent form an interval.  However, none of the four forbidden patterns have this property.  This is a contradiction.  Thus $v$ is Gasharov--Reiner.

Finally, it is easy to see that the first descent of $v$ is in position $i - 1$ and is a light reduction pair.
\end{proof}

Continuing with the notation of the preceding proof, we have that the first descent of $v(w)$ is between the entries $y' = (i - 1, w_i - 1)$ and $x' = (i, w_j - 1)$.  Since this descent is a light reduction pair, we may apply \eqref{LRP prop equation for matrices} to conclude that
\begin{equation}
\label{matrix LRP recursion on v}
\m_{v(w)}(q) = q \m_{s_{i - 1}v(w)}(q) + q^{n - 2}\m_{v(w) - y'}(q).
\end{equation}
It is easy to check that
\[
v(s_i w)  
= s_{i-1} \cdot v(w) , \qquad
v(w - y) 
= v(w) - y',        
\quad \text{ and } \quad
s_{i - 1}(w-y) 
=  s_iw - y'.
\]
Thus, we may multiply
\eqref{matrix LRP recursion on v} through by $q^n$ and substitute from~\eqref{eq:equations for vs} to conclude that
\begin{equation} 
\label{final HRP matrix recursion}
\m_w(q) - \m_{s_iw}(q) = q(\m_{s_iw}(q) - \m_{s_{i-1}s_iw}(q)) + q^{n-1}(\m_{w-y}(q) - \m_{s_iw - y'}(q)).
\end{equation}

Now we derive the same recursion for Poincar\'e polynomials.  The first descent of $s_iw$ is a heavy reduction pair involving the entries $y' = (i - 1, w_i - 1)$ and $x'' = (i, w_{i + 1})$, so by \eqref{eq:Poincare HRP recursion} we have
\begin{multline}
\label{eq:HRPsw}
P_{s_iw}(t) = P_{s_{i - 1}s_i w}(t) + t^{\inv(s_iw) - \inv(s_iw - x'')} P_{s_iw - x''}(t) \\+ t^{\inv(s_iw) - \inv(s_iw - y')}P_{s_iw - y'}(t) - t^{\inv(s_i w) - \inv(s_iw - x'' - y')}P_{s_iw - x'' - y'}(t).
\end{multline}
It is easy to check that 
\[
s_iw - x''  = w-x \qquad \text{ and } \qquad s_iw-x''-y' = w-x-y,
\]
and so subtracting $t$ times \eqref{eq:HRPsw} from \eqref{eq:Poincare HRP recursion} (keeping in mind that $\inv(s_iw) = \inv(w) - 1$) yields
\[
P_w(t) - tP_{s_iw}(t) = \left(P_{s_iw}(t) - t P_{s_{i-1}s_iw}(t)\right) 
+ t^{\inv(w)} \cdot (t^{- \inv(w- y)}P_{w-y}(t) - t^{- \inv(s_iw - y')}P_{s_iw - y'}(t)).
\]
Finally, we make the substitution $t = q^{-1}$ and multiply through by $q^{\binom{n}{2} + \inv(w)}$ to conclude
\begin{multline*}
q^{\binom{n}{2} + \inv(w)} P_w(q^{-1}) - q^{\binom{n}{2} + \inv(s_iw)} P_{s_iw}(q^{-1})
= \\ q\left(q^{\binom{n}{2} + \inv(s_i w)}P_{s_iw}(q^{-1}) -q^{\binom{n}{2} + \inv(s_{i - 1}s_i w)}  P_{s_{i-1}s_iw}(q^{-1})\right) \\
+ q^{n - 1}\left(q^{\binom{n - 1}{2}+ \inv(w- y)}P_{w-y}(q^{-1}) - q^{\binom{n - 1}{2} + \inv(s_iw - y')}P_{s_iw - y'}(q^{-1})\right).
\end{multline*}
Comparing with \eqref{final HRP matrix recursion} and applying the inductive hypothesis gives the desired result.

\section{Further remarks and questions}
\label{further remarks}

\subsection{Bijective proofs}
\label{sec:bijective proofs}

One can give an alternate proof of the first part of Theorem~\ref{theorem:AO-RA} via a recursive argument: given $w \in \Sn{n}$, one produces permutations $w' \in \Sn{n}$ and $w'' \in \Sn{n - 1}$ such that the inversion graphs $G_{w'}$ and $G_{w''}$ are isomorphic respectively to the graphs that we get by deleting or contracting a particular edge in $G_w$, and also such that rook placements on the south-west diagrams $\RSW_{w'}$ and $\RSW_{w''}$ correspond naturally to rook placements on $\RSW_{w}$ where a particular cell respectively does not or does contain a rook.  (See Figure~\ref{figexthm1}, and \cite{LMFPSAC} for more details.)  Then the result follows from the deletion-contraction recursion for acyclic orientations.  In principle, this can be unravelled (noncanonically) to give a
bijection. Is it possible instead to give a single, explicit (i.e.,
nonrecursive) bijection between rook placements avoiding $\RSW_w$ and
acyclic orientations of $G_w$? C.f.~Appendix~\ref{sec:AOvsRP}.

\begin{figure}
\begin{center}
\subfigure[]{
\includegraphics{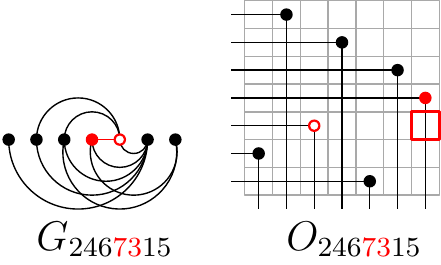}
\label{orig}
}
\quad
\subfigure[]{
\includegraphics{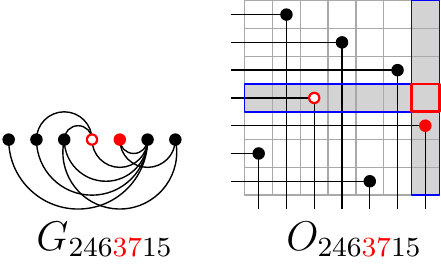}
\label{del}
}
\quad
\subfigure[]{
\includegraphics{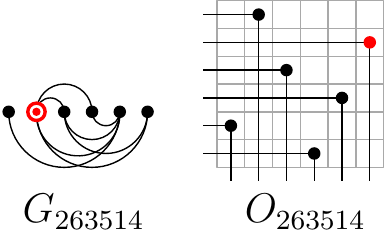}
\label{contract}
}
\caption{The inversion graph and SW diagram for (a) the permutation $2467315$, (b) the result of deletion, and (c) the result of contraction.}
\label{figexthm1}
\end{center}
\end{figure}



\subsection{Other types}

The acyclic orientations of the inversion graph of a permutation $w$ are in
correspondence with the regions of the hyperplane
arrangement $\mathcal{A}_w$ consisting of the hyperplanes $x_i - x_j =
0$ in $\mathbb{R}^{n}$ for each inversion $(i,j)$ of $w$. This arrangement has a
natural analogue when the symmetric group $\Sn{n}$ is replaced by any
Weyl group $W$. In this setting, Hultman \cite{H} has proved an analogue of 
Theorem~\ref{thm:HLSS} of Hultman--Linusson--Shareshian--Sj\"ostrand. Is
there an analogue of rook placements avoiding a
diagram associated to an element $w$ in $W$ that allows one to extend
Theorem~\ref{theorem:AO-RA} or Theorem~\ref{thm:matvspoin} to this
context? Barrese and Sagan (personal communication)  have made some
initial progress on this direction.

\subsection{A nicer recurrence for Poincar\'e polynomials}
As a consequence of Theorem~\ref{thm:matvspoin} and
Equation~\eqref{eq:matrix HRP recursion}, the Poincar\'e polynomial
for a Gasharov--Reiner permutation $w$ with a heavy reduction pair in
its first descent satisfies 
\[
P_w(t) = t\cdot P_{s_i w}(t) + P_{v(w)}(t).
\]
This latter recursion is arguably simpler than \eqref{eq:Poincare HRP
  recursion}. Is it possible to prove such a result directly, without
going through the painful contortions following the proof of
Proposition~\ref{prop:properties of v}? For example, can one exhibit a
bijection between the relevant Bruhat intervals that shifts lengths appropriately?

\subsection{Connection with Schubert varieties}

Theorem~\ref{thm:matvspoin} gives a relationship between a function
counting invertible matrices and a Poincar\'e polynomial.  The
Poincar\'e polynomial $P_w(t)$ has a geometric, as well as
combinatorial, meaning: it gives the decomposition of the Schubert
variety $X_w$ over $\CC$ into Schubert cells, or equivalently it
counts $\F_q$ points in $X_w$. In addition, the Gasharov--Reiner
permutations $w$ characterize the Schubert varieties $X_w$ {\em defined by
inclusions} \cite{GasharovReiner}.  (For an overview of connections between Schubert varieties and combinatorics, see \cite{AbeBilley}.)  Thus, it seems natural to suppose that there should be an explanation for Theorem~\ref{thm:matvspoin} involving the associated Schubert varieties.  At present, we have no such explanation for Gasharov--Reiner
permutations.  

However, it is possible to give a simple proof in the
special case of permutations avoiding the pattern $312$.  In this
case, the diagram $\RSW_w$ is a (reflection of a) Young diagram.  Thus
its complement $\overline{\RSW_w}$ is also a partition shape.  The
Schubert variety $X_w$ is one of Ding's \emph{partition varieties},
and Ding showed \cite[Thm.~33]{Ding} that the Poincar\'{e} polynomial
of this variety is equal to the Garsia--Remmel \emph{rook polynomial}
\cite{GarsiaRemmel}.  Next, work of Haglund \cite[Thm.~1]{Hag} shows
that for a partition shape, the rook polynomial and matrix count
$\m_w(q)$ are equal up to powers of $q$.  Finally, $312$-avoiding
permutations avoid the patterns $3412$ and $4231$, so by the work of
Lakshmibai--Sandya \cite{LakshmibaiSandhya} and Carrell--Peterson
\cite{CarrellPeterson} (Lemma~\ref{lemma:CP}) we may replace $P_w(q)$ with $q^{\inv(w)}P_w(q^{-1})$ to complete the proof.

The following result for smooth permutations follows easily from
Theorem~\ref{thm:matvspoin}; it was independently proven by
Linusson--Shareshian (personal
communication). Recall that $w \in \Sn{n}$ is smooth if $w$ avoids the
patterns $3412$ and $4231$.
\begin{cor} \label{cor:poinmatsmooth}
Let $w$ be a permutation in $\Sn{n}$.  We have
\[
\m_w(q) = q^{\binom{n}{2}}P_w(q)
\]
if and only if $w$ is smooth.
\end{cor}

\begin{proof}
To prove the ``if'' direction, first compare the definitions to see that if $w$ is smooth then $w$ is also Gasharov--Reiner. Thus it follows by Theorem~\ref{thm:matvspoin} that $\m_w(q)  = q^{\binom{n}{2}+\ell(w)} P_w(q^{-1})$. Then by Lemma~\ref{lemma:CP} we have that $q^{\binom{n}{2}+\ell(w)} P_w(q^{-1})= q^{\binom{n}{2}} P_w(q)$.

Next we prove the ``only if'' direction.  The argument in Section~\ref{sec:poinmat} following the statement of Theorem~\ref{thm:matvspoin} establishes that $w$ must be Gasharov--Reiner.  So, again using
Theorem~\ref{thm:matvspoin}, we have that $\m_w(q) =
q^{\binom{n}{2}+\ell(w)}P_w(q^{-1})$. This fact and the hypothesis
imply that $P_w(q)$ is palindromic, and by Lemma~\ref{lemma:CP} it
follows that $w$ is smooth.
\end{proof}

\subsection{Positivity in matrix-counting for other permutations}
Computational evidence suggests \cite[Conj.~5.1]{KLM} that
$\m_w(q)$ is a polynomial in $\mathbb{N}[q]$ for all permutations $w$,
not just for Gasharov--Reiner permutations.  (In general, the number of
invertible matrices over $\F_q$ with restricted support is not
necessarily a polynomial in $q$ \cite[\S  8.1]{Stem}.) It would be very interesting if one could explain this fact geometrically, e.g., via some sort of cellular decomposition of the set of matrices counted by $\m_w(q)$.  A more naive approach is to look for a 
recursion along the lines of Equations~\eqref{eq:matrix HRP recursion} and~\eqref{LRP
  prop equation for matrices} that is valid for all permutations. The next example gives some discouraging evidence for the latter approach.

\begin{example}
For $w=4312$, we have that 
\[
\m_{4312}(q) - \m_{3412}(q) = q^{11}+2q^{10}+2q^9-q^7,
\]
which has a negative coefficient.  Curiously, for all $w$ in $\Sn{n}$ for
$n\leq 7$ we have that $\m_w(q)-q\cdot \m_{s_iw}(q) \in \mathbb{N}[q]$.  However,
for $w=3412$, the difference 
\[
\m_{3412}(q) - q\cdot \m_{3142}(q) = 2q^8+3q^7+q^6
\]
is not of the form $q^a \cdot \m_{u}(q)$ for any integer $a$ and permutation $u$.
\end{example}

\begin{remark}
Note that if $w$ is \emph{not} Gasharov--Reiner then $\#[\id,w] > \AO(G_w)$ by
Theorem~\ref{thm:HLSS}.  The $q$-analogue of this fact is the following conjecture in \cite{KLM}:
for all $w$ the difference
$q^{\binom{n}{2}+\inv(w)}P_{w}(q^{-1})-\m_w(q)$ belongs to
$\mathbb{N}[q]$. On the other hand, there is no inequality of coefficients between $\m_w(q)$ and $q^{\binom{n}{2}}P_w(q)$ for the
non-smooth permutations. For example, $\m_{3412}(q)=q^{10}+3q^9+5q^8+4q^7+q^6$ and
$q^6P_{3412}(q)=q^{10}+4q^9+5q^8+3q^7+q^6$ are not comparable coefficientwise.
\end{remark}

\begin{remark}
For $n \leq 7$, computations show that the polynomials $\m_w(q)$ are \emph{unimodal} for all $w \in \Sn{n}$, i.e., their coefficients first increase, then decrease.  However, they are not generally \emph{log-concave}: when $w=5673412$ we have that the sequence of coefficients of $\m_w$ is $(1, 4, 17, 52, 116, 203, 289, 346, 355, 316, 246, 167, 98, 49, 20, 6, 1)$, and $4^2 < 1\cdot 17$ is a violation of log-concavity.
\end{remark}

\subsection{Matrices of lower rank}
By \cite[Prop.~5.1]{LLMPSZ}, the counting function for matrices of rank
$r$ with support avoiding a permutation diagram $\RSW_w$ is a
$q$-analogue of placements of $r$ non-attacking rooks avoiding
$\RSW_w$.  We conjecture \cite[Conj.~5.1]{KLM} that this function belongs to
$(q-1)^r\mathbb{N}[q]$; what are its coefficients counting?  Are there
corresponding ``lower rank'' analogues of any other members of
Postnikov's ``zoo'', either for all permutations or for some nice subclass (e.g., Grassmannian permutations, smooth permutations, Gasharov--Reiner permutations)?

\subsection{Counting and $q$-counting fillings of permutation diagrams} \label{sec:fillings}

Above, we have studied percentage-avoiding fillings of the SE diagram $\RSE_w$ for a permutation $w$.  When $w = w_\lambda$ is Grassmannian, $\RSE_w$ is the Young diagram of $\lambda$ in French notation.  For such shapes, percentage-avoiding fillings are in bijection with a large family of similarly restricted fillings (see \cite{AS, MJV}), including the \emph{$\sammag$-fillings}%
\footnote{
In \cite{AP}, Postnikov used English notation for partitions, while we
use French notation; thus, his ``$\sLe$ diagrams'' are equivalent to our
$\sammag$-fillings and what would be his ``$\sEl$ diagrams'' 
correspond to our $\Gamma$-fillings. } studied by Postnikov.  Here, we mention some additional results and conjectures relating to these other restricted fillings of the diagram $\RSE_w$ when $w$ is not necessarily Grassmannian.

Given a binary filling $f$ of a diagram $D$, we say that $f$ is a $\sLe$-filling if it avoids the patterns \raisebox{-5pt}{\includegraphics{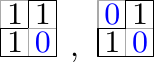}}.  Similarly, we say that $f$ is
\begin{compactitem}
\item a $\sGamma$-filling if it avoids the patterns \raisebox{-5pt}{\includegraphics{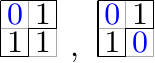}},
\item a $\sEl$-filling if it avoids the patterns \raisebox{-5pt}{\includegraphics{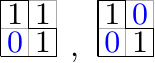}}, and
\item a $\sammag$-filling if it avoids the patterns 
  \raisebox{-5pt}{\includegraphics{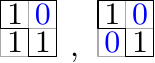}}.
\end{compactitem}

In this section, we focus on the case that the diagram $\RSE_w$ has a particular nice structure.  We say that a diagram $D$ has the {\bf south-east (SE) property} if whenever
$(i,j)$, $(i',j)$ and $(i,j')$ are in $D$ with $i'>i$ and $j'>j$ then
$(i',j')$ is also in $D$.  For such diagrams, condition~(ii) in Definition~\ref{PPAF definition} is never relevant, and so pseudo-percentage avoidance reduces to percentage avoidance in this case.

It is easy to see that the SE diagram $\RSE_{321}$ fails to have the SE property, and consequently that $\RSE_w$ also fails to have the SE property for any permutation $w$ containing $321$ as a pattern.  The converse of this statement is also true:

\begin{proposition}[{\cite[Prop.~2.2.13]{LM}}]\label{prop:Ew321}
If $w$ avoids $321$ then $\RSE_w$ is, up to removing rows and columns that do not intersect
$\RSE_w$, a skew Young shape in French notation. In this case $\RSE_w$
has the SE property.
\end{proposition}

We start by giving a corollary of
Proposition~\ref{proposition:AO-PAF}.

\begin{cor}
\label{corollary:AO-Le}
If $w$ avoids $321$ then the number of $\sGamma$-fillings of $\RSE_w$ is equal to the number of acyclic
orientations of the inversion graph of $w$.
\end{cor}

\begin{proof}
By Proposition~\ref{proposition:AO-PAF}, for all $w$ the number of acyclic
orientations of the inversion graph of $w$ equals the number of
pseudo-percentage-avoiding fillings of $\RSE_w$.  
By Proposition~\ref{prop:Ew321} and the paragraphs that precede it, if $w$ avoids $321$ then a filling of $\RSE_w$ is pseudo-percentage-avoiding if and
only if it is percentage-avoiding.  Moreover, since $\RSE_w$ is a skew Young
shape, we have by work of Spiridonov \cite{AS}
(see also Josuat-Verg\`es \cite[\S 4]{MJV}) that the number of
percentage-avoiding fillings and the number of $\sGamma$-fillings of
$\RSE_w$ are equal.
\end{proof}

\begin{remark} 
\label{rem:diff321}
\label{ex:diffLeEl}
Although the number of $\sEl$-fillings and
$\sGamma$-fillings of $\RSE_w$ coincide when $w$ is Grassmannian, these numbers
can differ for other $321$-avoiding permutations.  For example,
the permutation $351624$ avoids $321$, and one can check
that there are $98$ $\sGamma$-fillings and $100$ $\sEl$-fillings of
$\RSE_{351624}$. By Theorem~\ref{prop:BruhatSkew} below, the latter
are in bijection with the elements in the interval $[\id,351624]$.
\end{remark}

\begin{figure}
\begin{center}
\subfigure[]{
\includegraphics{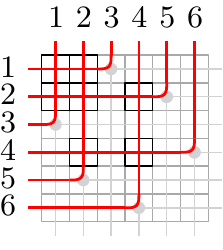}
\label{wd}
}
\quad
\subfigure[]{
\includegraphics[height=2.5cm]{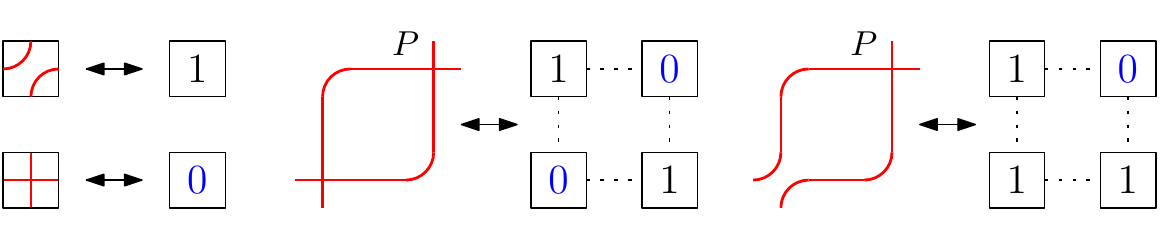}
\label{wdel}
}
\caption{(a)  The hook wiring diagram of $w=351624$ which has
  crossings in exactly the elements of $E_{351624}$, (b)
  correspondence between (un)crossings in a pipe dream and
 binary fillings; the forbidden wires for lexicographic-maximal
  pipe dreams and the corresponding $\sammag$-patterns.}
\label{wiringdiagram:ammag}
\end{center}
\end{figure}

We now focus on $\sEl$-fillings and $\sammag$-fillings.
Given a binary filling $f$ of a diagram, the size $|f|$
is the number of $1$s in the filling. 
For $w$ in $\Sn{n}$ and $\pi \in \{\sEl,\sammag\}$, let
$F_w^{\pi}(q)$ be the generating function $\sum_{f} q^{|f|}$ where
the sum is over $\pi$-fillings of $\RSE_w$. By abuse of notation,
$F^{\pi}_{\lambda}(q) = F^{\pi}_{w_\lambda}(q)$ is the generating function of
$\pi$-fillings of the Young diagram of the partition $\lambda$.  The following is a corollary of \cite[Thm.~19.1]{AP}.

\begin{theorem}[Postnikov] \label{thm:postElGrass}
For a Grassmannian permutation $w_{\lambda}$ in $\Sn{n}$ associated to a
partition $\lambda \subseteq k \times (n-k)$, we have that
$F^{\sEl}_{\lambda}(q) = F^{\scrammag}_{\lambda}(q)=q^{\ell(w_{\lambda})}P_{w_{\lambda}}(q^{-1})$. 
\end{theorem}

\begin{proof}[Proof sketch]
Given $w$
in $\Sn{n}$, fix a reduced decomposition  of $w$ and the corresponding
{\em wiring diagram} of the decomposition. Then each $u$ in $[\id,w]$
 is obtained as a subword of the reduced decomposition \cite[Prop.~2.1.3]{LM}. To rule out repetitions one can choose the
{\em lexicographically maximal} ({\em minimal}) subword that is a reduced
expression for $u$. Postnikov then characterized these subwords as
certain {\em pipe dreams} of the wiring diagram, obtained by changing
crossings of wires to uncrossings, with two restrictions: if two wires cross at a point $P$ then they cannot cross
or uncross before (after) $P$. Call these 
{\em lexicographically maximal (minimal) pipe dreams}; see Figure~\ref{wdel}. It follows that they are in bijection with the
elements $u$ in $[\id,w]$.

Next, for a Grassmannian permutation $w_{\lambda}$ in 
$\Sn{n}$, Postnikov describes a wiring diagram with crossings exactly
on the cells of the Young diagram of $\lambda$. Then pipe dreams of this
wiring diagram correspond to fillings of $\lambda$. 
Moreover, the lexicographically maximal
(minimal) pipe dreams are exactly the
$\sammag$-fillings ($\sEl$-fillings) of $\lambda$. This yields a
correspondence $u\leftrightarrow f$ between $u$ in $[\id,w_{\lambda}]$ and
$\sammag$-fillings ($\sEl$-fillings) $f$ of $\lambda$ such that $\ell(u) = \ell(w) - |f|$, as desired.
\end{proof}

Theorem~\ref{thm:postElGrass} can be extended to $321$-avoiding permutations.  (This extension is due to Postnikov--Spiridonov (personal communication), but it seems that the statement has not been written down anywhere before.)

\begin{theorem} \label{prop:BruhatSkew}
For $w$ in $\Sn{n}$ avoiding $321$
we have that  $F^{\sEl}_{w}(q) = F^{\scrammag}_{w}(q)=q^{\ell(w)}P_{w}(q^{-1})$.
\end{theorem}

\begin{proof}[Proof sketch]
The argument is essentially the same as
that of Theorem~\ref{thm:postElGrass} sketched
above. It is necessary to give a wiring diagram for $w$
analogous to the one for a  Grassmannian permutation
with crossings exactly on the cells of the diagram of $w$. Given $w$, for each $i=1,\ldots,n$ we draw a
wire starting from the first entry of the $i$th row that goes right
until it reaches the entry $(i,w_i)$ where it turns $90^{\circ}$ and
continues up to end in the first entry of the $w_i$th column. This collection of $n$ wires is a wiring diagram of $w$
with crossings in exactly the elements $(i,j) \in \RSE_w$. See
Figure~\ref{wd} for an example and \cite[Rem.~2.1.9]{LM} for a similar
construction. We call this wiring diagram the {\bf
  hook wiring diagram} of $w$. If $w$ is a Grassmannian permutation
$w_{\lambda}$ or $w$ avoids $321$ then $\RSE_w$ is, up to
removing rows and columns that do not intersect $\RSE_w$,
the Young diagram of $\lambda$ or of a skew Young shape respectively. The rest of the argument in the
proof of  Theorem~\ref{thm:postElGrass} follows for this wiring diagram
on the skew shape. However, the argument can fail for $w$ containing $321$ (see Figure~\ref{fig:lexpipedreams4231}).
\end{proof}

\begin{remark}
A consequence of this result and Lemma~\ref{lemma:CP} is that when $w$ avoids
$321$ and $3412$ then $F^{\sEl}_w(q)=F^{\scrammag}_w(q) =
P_w(q)$. From computations, both statements appear to be if-and-only-ifs. This class of permutations has also appeared several times
in the literature \cite{PetTenn,pattdatabase}.
\end{remark}

\begin{figure}
\begin{center}
\subfigure[]{
\includegraphics{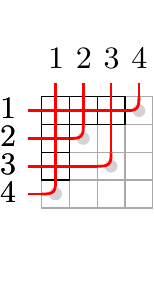}
}
\quad
\subfigure[]{
\includegraphics{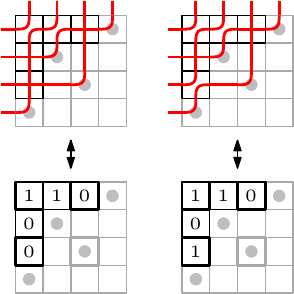}
\label{fig:ammagpipes}
}
\quad
\subfigure[]{
\includegraphics{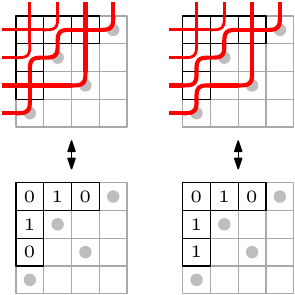}
\label{fig:lexmaxpipes}
}
\caption{(a) The hook wiring diagram of $w=4231$ associated to the reduced
word $s_3s_2s_1s_2s_3$. (b) Two fillings of $\RSE_{4231}$ containing the pseudo-$\sammag$ pattern that correspond to lexicographically maximal pipe
  dreams. (c) Two lexicographically non-maximal pipe dreams of the hook  wiring diagram of  $4231$ that do correspond to pseudo-$\sammag$ fillings of
  $\RSE_{4231}$.}
\label{fig:lexpipedreams4231}
\end{center}
\end{figure}

\subsection{$q$-counting pseudo fillings of permutation diagrams} \label{sec:pseudofillings}
In this section we look briefly at fillings of $\RSE_w$ where the
diagram might not have the SE property.  Because of this defect, we put extra restrictions on the fillings just as we did with the percentage avoiding fillings in
Section~\ref{sec:AOvsPAF}. We say that a filling $f$ of $\RSE_w$ is
\begin{compactitem}
\item  a \textbf{pseudo-$\sEl$-filling} if it avoids the patterns
\raisebox{-7pt}{\includegraphics{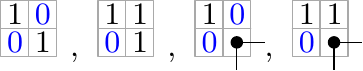}},  where the solid
dot indicates an entry of the permutation, and
\item a \textbf{pseudo-$\sammag$-filling} if it avoids the patterns
\raisebox{-7pt}{\includegraphics{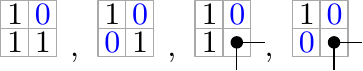}}, where the solid
dot indicates an entry of the permutation.
\end{compactitem}
For $w$ in $\Sn{n}$ and $\pi \in \{\sEl,\sammag\}$, let
$\PF_w^{\pi}(q)$ be the generating function $\sum_{f} q^{|f|}$ where
the sum is over fillings of $\RSE_w$ avoiding the appropriate
pseudo-$\pi$ pattern. Note that if $\RSE_w$ has the SE property then
the last two patterns to avoid  in pseudo-$\pi$-fillings
will never be relevant, and so these
fillings reduce to the usual $\pi$-fillings. For such $w$ we have that $\PF^{\pi}_w(q)=F^{\pi}_w(q)$.

The next conjecture
suggests an extension of Theorem~\ref{prop:BruhatSkew} for Gasharov--Reiner
permutations and pseudo-fillings.

\begin{conjecture} \label{conj:pseudofillpoin}
For $w$ in $\GR{n}$ we have that 
\[
\PF^{\sEl}_w(q) = \PF^{\scrammag}_w(q) =  q^{\ell(w)}P_{w}(q^{-1}).
\]
\end{conjecture}
\noindent This conjecture has been verified by brute force for $n\leq
7$. A proof of this conjecture, combined with Theorems~\ref{thm:HLSS}
and \ref{theorem:AO-RA}, would
extend the equivalence of Theorem~\ref{thm:APGrass} from
Grassmannian to
Gasharov--Reiner permutations. Recall that the combinatorial objects
in Theorem~\ref{thm:APGrass} identified with Grassmannian permutations
$w_{\lambda}$ also count and parametrize {\em positroid cells} inside a {\em Schubert cell} $\Omega_{\lambda}$. Do some of the objects
described in this paper linked to other permutations $w$ count cells in a
decomposition of a generalization of $\Tnn$?

\begin{remark}
Note that the number of pseudo-$\sEl$-fillings and the number of
pseudo-$\sammag$-fillings of $\RSE_w$ can differ for certain
permutations $w$. For example for $w=35241$ the
number of pseudo-$\sEl$ fillings of $\RSE_{35241}$ is $56$ and the
number of pseudo-$\sammag$ fillings of $\RSE_{35241}$ is $60$. The
numbers differ also for the inverse  $53142$ of $w$. These are the only permutations in $\Sn{5}$ where
the number of these two fillings differ. 

Similarly, the number of pseudo-$\sEl$-fillings of $\RSE_w$ and the size $\#[\id,w]$
of the Bruhat interval can differ for certain permutations $w$. For
example, for $w=52341$, we have that $\PF^{\sEl}_{52341}(1)=72$ and $\#[\id,52341]=68$. 
\end{remark}

\begin{remark}
One approach to prove Conjecture~\ref{conj:pseudofillpoin} would be to 
extend Postnikov's correspondence from Theorem~\ref{thm:postElGrass}
to lexicographically maximal (minimal) pipe dreams encoding $u$ in
$[\id,w]$ and pseudo-$\sammag$- (pseudo-$\sEl$-) fillings of $\RSE_w$. Brute
force calculations suggest there is such a correspondence for all $w$
in $\GR{n}$ up to $n\leq 6$ but, it may
fail for other permutations; see Figure~\ref{fig:lexpipedreams4231}.

Another approach to prove the conjecture is the reduction pairs used in the proof of
Theorem~\ref{thm:matvspoin}. One can show, using an analysis similar to the one by Williams in
\cite{LWilliams}, that if the first descent of $w$, involving the entries
$y=(i,w_i)$ and $x=(i+1,w_{i+1})$, is a light reduction pair then
\begin{equation}
\PF_w^{\scrammag}(q) = q\cdot \PF_{s_iw}^{\scrammag}(q) + \PF_{w-y}^{\scrammag}(q),
\end{equation}
and it is also not difficult to show that if the first descent of $w$, involving the
entries $y=(i,w_i)$ and $x=(i+1,w_{i+1})$, is a heavy reduction pair then
\begin{equation}
\PF_w^{\sEl}(q) = q\cdot \PF_{s_iw}^{\sEl}(q) + \PF_{w-y}^{\sEl}(q) + \PF_{w-x}^{\sEl}(q) - \PF_{w-y-x}^{\sEl}(q).
\end{equation}
These recursions match those for $P_w(q)$ in Propositions~\ref{prop:LRP} and~\ref{prop:HRP}; however, we have been unable to prove the two corresponding recursions necessary to complete the induction.  We have also been unable to prove $\PF_w^{\pi}(q) =
\PF^{\pi}_{s_iw}(q) + q\cdot PF_v^{\pi}(q)$ for $\pi$ in $\{\sEl, \sammag\}$, which would be analogous to \eqref{eq:matrix HRP recursion}.
\end{remark}

\section*{Acknowledgements}
\label{sec:ack} 
We thank Alexander Postnikov for multiple suggestions and questions
that led to this project; Axel Hultman for writing an appendix with an
elegant proof of part of Theorem ~\ref{thm:AO-RP-PAF}; Ricky Liu, Luis Serrano, and Alexey Spiridonov for very helpful
discussions;  Herman Goulet-Ouellet for bringing to our attention the
paper \cite{GJW1978}; and Pasha Pylyavskyy for the term ``zoo'' in the
introduction. We thank the referees for helpful comments and suggestions.
We also thank the developers of \rm{FindStat} \cite{findstat}, from
which we obtained the first evidence for
Theorem~\ref{thm:AO-RP-PAF}. 
The second author would like to thank ICERM and the organizers of its
program in {\em Automorphic Forms, Combinatorial
  Representation Theory, and Multiple Dirichlet series} during which 
part of this work was done. 

\appendix
\section{Acyclic orientations, rook placements, inversion graphs and the chromatic polynomial (by Axel Hultman)}
\label{sec:AOvsRP}

In this appendix we provide an independent proof of the part of Theorem
\ref{theorem:AO-RA} which asserts $\AO(G_w) = \RP(\RSW_w)$. More
precisely, we prove the following statement:

\begin{theorem}\label{th:chromatic}
For any $w\in \Sn{n}$, the chromatic polynomial of the inversion
graph $G_w$ satisfies
\begin{equation} \label{eq:chromatic}
\chi_{G_w}(t) = \sum_{i=0}^n r_{n-i}(\RSW_w) t(t-1)\cdots (t-i+1),
\end{equation}
where the {\em rook number} $r_k ({\RSW_w})$ is the number of placements of
$k$ non-attacking rooks on $\RSW_w$.
\end{theorem}

From  (\ref{eq:chromatic}), the desired assertion follows if one sets $t=-1$
and invokes the standard results
\begin{equation}\label{eq:acy}
\AO(G_w) = (-1)^n\chi_{G_w}(-1)
\end{equation}
and
\begin{equation}\label{eq:hit}
\RP(\RSW_w) = \sum_{i=0}^n(-1)^i r_i (\RSW_w)(n-i)!.
\end{equation}
The identity (\ref{eq:acy}) was originally obtained by Stanley
\cite{Stanley1973} whereas (\ref{eq:hit}) is due to Kaplansky and Riordan \cite{KR1946}.

The idea behind the proof of Theorem \ref{th:chromatic} is
essentially that employed by Goldman, Joichi and White for
proving \cite[Theorem 2]{GJW1978}. Some care is required, though, since
$\RSW_w$ is not in general {\em proper} in the sense of
\cite{GJW1978}. It is, however, possible to make it proper by a suitable rearrangement of
its columns. Then one could apply \cite[Theorem 2]{GJW1978}
directly. After observing that the associated graph $\Gamma_n(B)$ of the
rearranged board is isomorphic to $G_w$, Theorem \ref{th:chromatic}
would follow. Instead of taking this route, let us state a direct proof.

\begin{proof}[Proof of Theorem \ref{th:chromatic}]
In a graph $G$, a subset of the vertices is called {\em independent} if it induces an edgeless
subgraph of $G$. For a positive integer $k$, denote by $P(w,k)$ the
set of partitions of the vertex set of the inversion graph $G_w$ into
$k$ independent subsets. Equivalently, we may think of $P(w,k)$ as the
set of transitively closed subgraphs of the complement graph
$\overline{G_w}$ with $k$ connected components and
all $n$ vertices.

Let us say that an {\em $n$-spine} is a graph on vertex set $[n]$ in which
every connected component is a path whose vertices can be
traversed in increasing (or, going the other way, decreasing)
order. Equivalently, a graph on $[n]$ is an $n$-spine if every vertex
has at most one smaller neighbour and at most one larger neighbour.

A rook on $\RSW_w$ corresponds to a noninversion of $w$, i.e.\
an edge of $\overline{G_w}$. In this way, the non-attacking rook
placements on $\RSW_w$ are in bijective correspondence with the sets
of edges of $\overline{G_w}$ that contain no two edges with a common
smallest vertex and no two edges with a common largest vertex. That
is, the non-attacking $k$-rook placements on $\RSW_w$ correspond
bijectively to the $k$-edge subgraphs of $\overline{G_w}$ that are
$n$-spines. 

If $1\le i_1<i_2<i_3\le n$ and $(i_1,i_2)$ and $(i_2,i_3)$ are noninversions of $w$,
then so is $(i_1,i_3)$. Hence, the transitive closure of any $k$-edge $n$-spine
subgraph of $\overline{G_w}$ is an element of $P(w,n-k)$. Conversely,
every element of $P(w,n-k)$ is clearly the closure of a unique
$n$-spine. This shows that $P(w,n-k)$, too, is in bijection with the $k$-edge
subgraphs of $\overline{G_w}$ that are $n$-spines. Hence,
$r_{n-i}(\RSW_w) = \#P(w,i)$. Now observe that
\[
\chi_{G_w}(t) = \sum_{i=0}^n \#P(w,i) \cdot t(t-1)\cdots (t-i+1),
\]
since (for a positive integer $t$) the term indexed by $i$ in the sum
counts the proper vertex colourings of $G_w$ that use exactly $i$ out
of $t$ given colours. This concludes the proof.
\end{proof}

An illustration of the constructions occurring in the proof is found in
Figure \ref{fi:bijections}.

\begin{figure}[t]

\includegraphics[viewport = 134 550 537 707, clip=true, scale=.85]{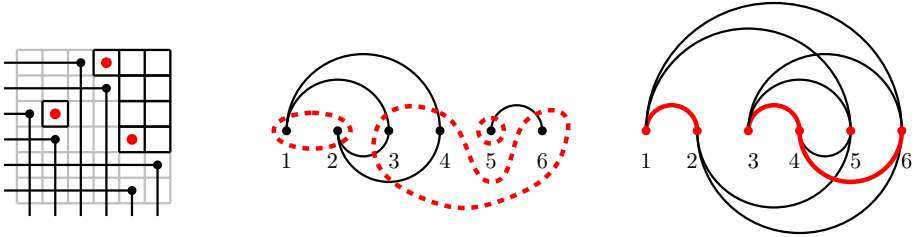}

\caption{A non-attacking $3$-rook placement on the SW diagram of
  $w=341265 \in \Sn{6}$ (left), the corresponding partition of $G_w$
  into $6-3=3$ independent sets (center) and the associated $3$-edge
  subgraph of $\overline{G_w}$ which forms a $6$-spine (right).} \label{fi:bijections}
\end{figure}

\begin{remark}[by AHM and JBL]
\label{vexillary case of AO-RA}
The equality \eqref{eq:chromatic} is particularly nice
when the reverse of $w$ is \emph{vexillary}, i.e., when $w$ avoids $3412$. In
this case $\RSW_w$ is, up to permuting rows and
columns, a Young diagram $\lambda =
(\lambda_1,\lambda_2,\ldots,\lambda_n)$ where $0\leq \lambda_1\leq \lambda_2
\leq \cdots$.  Then calculating the right side of
\eqref{eq:chromatic} is straightforward: by \cite{GJW1975} we have that $\sum_{i=0}^n r_{n-i}(\RSW_w) t(t-1)\cdots (t-i+1)
= \prod_{i=1}^n (t +\lambda_i -i+1)$.  On the other hand, we say that a graph is {\em chordal} if every cycle of four or more edges in the
graph has a chord, i.e., an edge joining two non-consecutive vertices in
the cycle. It is well known that the chromatic polynomial of a chordal graph $G$ may be written as $\prod_{i=1}^n (t-e_i)$ for
certain nonnegative integers $e_i$ depending on $G$ (see e.g. \cite[Prop.~12]{OPY}). One can show that the
inversion graph $G_w$ is chordal if and only if $w$ avoids
$3412$ and that in this case the multisets $\{e_i\}_{i=1}^n$ and $\{i-\lambda_i
-1\}_{i=1}^n$ are equal. 
\end{remark}

\newcommand{\etalchar}[1]{$^{#1}$}

\bibliographystyle{alpha}

\end{document}